\documentclass[final,onefignum,onetabnum]{siamart171218}

\usepackage{lipsum}
\usepackage{amsfonts}
\usepackage{amsmath}
\usepackage{mathrsfs}
\usepackage{amssymb}
\usepackage{booktabs}
\usepackage{graphicx}
\usepackage{epstopdf}
\usepackage{algorithmic}
\usepackage{mathtools}
\ifpdf
  \DeclareGraphicsExtensions{.eps,.pdf,.png,.jpg}
\else
  \DeclareGraphicsExtensions{.eps}
\fi

\definecolor{HsL}{RGB}{255,0,255}

\newcommand{\stf}{\mathrm{stf}}
\newcommand{\Kn}{\mathrm{Kn}}
\newcommand{\dif}{\mathrm{d}}
\newcommand{\pdfFrac}[2]{\frac{\partial #1}{\partial #2}}
\newcommand{\innerprod}[2]{\left\langle#1,#2\right\rangle}

\newcommand{\odd}{\mathrm{odd}}
\newcommand{\even}{\mathrm{even}}

\newcommand{\ext}{\mathrm{ext}}
\newcommand{\bm}[1]{\boldsymbol{#1}}
\newcommand{\bcal}[1]{\boldsymbol{\mathcal{#1}}}

\newcommand{\qbarb}{\boldsymbol{\bar{s}}}

\newcommand{\qb}{\boldsymbol{s}}
\newcommand{\sigmabarb}{\boldsymbol{\bar{\sigma}}}
\newcommand{\sigmab}{\boldsymbol{\sigma}}

\newcommand\bxi{\boldsymbol{\xi}}
\newcommand{\bbm}[1]{\bar{\boldsymbol{#1}}}
\newsiamremark{remark}{Remark}
\newsiamremark{hypothesis}{Hypothesis}
\crefname{hypothesis}{Hypothesis}{Hypotheses}
\newsiamthm{claim}{Claim}

\headers{Wellposedness for R13}{Shuang Hu, Bo Lin, Huini Liu and Zhenning Cai}
\title{Global Well-Posedness of the Linearized R13 Moment Equations with Onsager Boundary Conditions\thanks{Submitted to the editors DATE.
\funding{The work of Zhenning Cai was supported by the Academic Research Fund of the Ministry of Education of Singapore under Grant No. A-8002392-00-00.}}}

\author{Shuang Hu\thanks{
Zhejiang University, 866 Yuhangtang Road, Hangzhou, Zhejiang Province, 310058 China
  (\email{hsmathna@gmail.com}).}
\and Bo Lin\thanks{Beijing Huairou Laboratory, 8 Yangyan East 1st Road, Beijing 101400, China.
  (\email{linbo@neps.hrl.ac.cn}).}
\and Huini Liu\thanks{Hengyang Normal University, 16 Henghua Road, Hengyang, Hunan Province, 421002, China.  
  (\email{liuhuini@smail.xtu.edu.cn}).}  
\and Zhenning Cai\thanks{ Corresponding author. 
Department of Mathematics, National University of Singapore,
Level 4, Block S17, 10 Lower Kent Ridge Road, Singapore 119076  (\email{matcz@nus.edu.sg}).}}

\usepackage{amsopn}

\makeatletter
\newcommand*{\addFileDependency}[1]{
  \typeout{(#1)}
  \@addtofilelist{#1}
  \IfFileExists{#1}{}{\typeout{No file #1.}}
}
\makeatother

\ifpdf
\hypersetup{
  pdftitle={Global Well-Posedness of the Linearized R13 Moment Equations with Onsager Boundary Conditions},
  pdfauthor={Shuang Hu, B. Lin}
}
\fi

\begin{document}

\maketitle

\begin{abstract}
This paper establishes the global well-posedness of the linearized regularized 13-moment (R13) equations for rarefied gas flows. We first derive an entropy inequality for the system on bounded domains subject to Onsager wall boundary conditions. For the steady-state problem, well-posedness is proved via the Ladyzhenskaya–Babuška–Brezzi (LBB) theorem, facilitated by novel boundary-related Korn-type inequalities. Furthermore, leveraging the Lumer-Phillips theorem, we extend these results to guarantee the global well-posedness of the time-dependent R13 equations. Our theoretical framework uniformly accommodates the models for both Maxwell and general non-Maxwell molecules.
\end{abstract}



\section{Introduction} 
Rarefied gas flows are characterized by non-equilibrium effects that emerge when the molecular mean free path is comparable to the characteristic length scale of the system, a regime parameterized by the Knudsen number $\Kn$. 
For flows in this regime, the classical Navier--Stokes--Fourier system might fail to provide an accurate description, while the Boltzmann equation, the fundamental equation of gas kinetic theory, offers a complete statistical description via the molecular velocity distribution function. Although significant advances have been made in the development of robust numerical solvers for the Boltzmann equation \cite{gamba2026cmame,pareschi2026sinum,lin2024sinum,pinn2025sisc,yin2026sisc,wu2024cnf}, its intrinsic high-dimensional phase-space structure makes the direct numerical solution computationally expensive for practical applications. 
Moment methods provide an efficient bridge between these two descriptions by projecting the Boltzmann equation onto a finite set of macroscopic variables, 
leading to closed systems of continuum-type partial differential equations (PDEs) that retain essential non-equilibrium kinetic information, while avoiding the curse of dimensionality associated with solving the full Boltzmann equation directly \cite{Struchtrup2005macroscopic}.

The moment method was pioneered by Grad \cite{grad1949kinetic}, who introduced finite dimensional systems involving macroscopic fields such as the density $\rho$, velocity $\bm{u}$, temperature $\theta$, stress tensor $\bm{\sigma}$, and heat flux $\bm{s}$. However, Grad's original moment systems suffer from well-known stability and hyperbolicity difficulties. To improve the physical and mathematical behavior of the 13-moment closure, Struchtrup and Torrilhon derived the regularized 13-moment (R13) equations \cite{struchtrup2003regularization}. For boundary value problems, thermodynamically admissible Onsager boundary conditions were later formulated by Rana and Struchtrup \cite{RanaStruchtrup2016}, and their connection with stability properties of moment systems was further clarified by Torrilhon and Sarna \cite{TorrilhonSarna2017}.
Recent works have further derived linearized R13 models equipped with stable Onsager boundary conditions for both Maxwell and non-Maxwell molecules \cite{cai2024linear,lin2025time}, as well as evaporation boundary conditions \cite{BeckmannRanaTorrilhonStruchtrup2018}.

Previously, the well-posedness theory for linearized R13 moment systems had remained relatively underdeveloped, which limited the rigorous analysis of related numerical schemes. Existing numerical studies include its discontinuous Galerkin discretization \cite{TorrilhonSarna2017}, the continuous interior penalty finite element methods for the steady linear R13 system \cite{westerkamp2019finite,Torrilhon2021FEniCS}, and fundamental solution techniques \cite{Himanshi2025generalized}. In the absence of a well-established continuous theory, convergence analysis for these schemes was difficult, and practical implementations often relied on problem-dependent stabilization terms. A recent work of Lewintan et al. \cite{lewintan2025wellposedness} addressed this issue for the R13 model studied in \cite{Torrilhon2021FEniCS} by combining the LBB theorem \cite{boffi2013mixed} with generalized Korn-type inequalities \cite{ciarlet2010korn,dain2006generalized,gmeineder2024limiting}. This result clarifies the saddle-point structure of the linearized R13 equations and provides a useful analytical basis for subsequent numerical analysis. Building on this perspective, Hu et al. \cite{hu2026stabilityconvergencemixedfinite} developed a stable mixed finite element method without continuous interior penalty terms.

The above developments, however, leave several aspects of the well-posedness theory unresolved. The functional setting in \cite{lewintan2025wellposedness}, while effective for the continuous problem, imposes strong regularity requirements on some variables and may lead to mismatched convergence rates in finite element discretizations \cite{hu2026stabilityconvergencemixedfinite}. Moreover, the analysis there is restricted to Maxwell molecules and to the steady-state problem. A unified treatment covering more general intermolecular potentials and the time-dependent R13 system is therefore still needed.

To address these issues, we establish a unified well-posedness theory for the linearized R13 system with Onsager boundary conditions, covering both Maxwell and non-Maxwell molecules. Our main contributions are as follows:
\begin{itemize}
    \item We derive an entropy inequality on bounded domains, showing that the linearized R13 system with Onsager boundary conditions is consistent with the second law of thermodynamics.

    \item We show the well-posedness of the steady-state problem. The non-Maxwell case is handled through boundary-related Korn-type inequalities and the LBB framework. In contrast, the Maxwell case requires a separate functional setting because the non-Maxwell coercivity mechanism is degenerate.

    \item We prove the global well-posedness of the time-dependent R13 system with homogeneous Onsager boundary data via semigroup theory and the Lumer--Phillips theorem, and briefly discuss the extension to non-homogeneous \\
    boundary data through a standard lifting argument.
\end{itemize}

The present well-posedness theory provides a useful continuous foundation for the future design and analysis of numerical schemes for the R13 system. The saddle-point formulations and coercivity estimates clarify the stability mechanisms that should be preserved in discretizations. In addition, the connection between the Maxwell and non-Maxwell cases suggests a possible direction for studying the asymptotic behavior of the general model toward the Maxwell regime.

The rest of the paper is organized as follows. The following subsection \ref{subsec:Notation} introduces the notation used throughout the paper. Section \ref{Sec:LR13} presents the linearized R13 system, from its distribution-function formulation of the moment equations. Section \ref{Sec:Entropy} derives entropy inequalities on bounded domains. Sections \ref{Sec:Steady_State} and \ref{Sec:TimeDependent} establish the well-posedness of the steady-state and time-dependent problems, respectively. Finally, Section \ref{Sec:Conclusion} concludes the paper and discusses future directions.
\subsection{Notations}
\label{subsec:Notation}
Let $\Omega \subset \mathbb{R}^3$ be a bounded Lipschitz domain with boundary $\Gamma$. We use standard Sobolev spaces $H^m(\Omega; \mathbb{X})$ for functions taking values in a finite-dimensional space $\mathbb{X}$ (scalars, vectors, or tensors), equipped with the norm $\|\cdot\|_{m,\Omega}$. The case $m=0$ recovers $L^2(\Omega; \mathbb{X})$ space, where its standard inner product is denoted by $(\cdot, \cdot)$. For functions with well-defined $L^2$ traces, the boundary inner product on $\Gamma$ is denoted by $(\cdot, \cdot)_b$. We use $\cdot$ and $:$ for the scalar products of vectors and second-order tensors, respectively. The notation $A \lesssim B$ (resp., $A\gtrsim B$) implies $A \le CB$ (resp., $A\ge CB$) for a generic constant $C > 0$ independent of the functions.

We adopt the Einstein summation convention over repeated indices. For a tensor, $(\cdot)_{\mathrm{stf}}$ or angle brackets $\langle \dots \rangle$ around its indices denote its symmetric and trace-free (stf) part. Specifically, for a second-order tensor $A$ and a third-order tensor $B$, their stf components are defined as:
\begin{equation*}
A_{\langle ij \rangle} = A_{(ij)} - \frac{1}{3}A_{kk}\delta_{ij}, \quad B_{\langle ijk \rangle} = B_{(ijk)} - \frac{1}{5}\big(B_{(ill)}\delta_{jk} + B_{(ljl)}\delta_{ik} + B_{(llk)}\delta_{ij}\big),
\end{equation*}
where $A_{(ij)}$ and $B_{(ijk)}$ denote the fully symmetric parts obtained by averaging over all index permutations \cite{westerkamp2019finite}, and $\delta_{ij}$ is the Kronecker delta.

In the kinetic context, let $f_{\mathcal{M}}(\boldsymbol{\xi}) := (2\pi)^{-3/2} \exp(-|\boldsymbol{\xi}|^2/2)$ denote the normalized global Maxwell distribution. We consider microscopic distribution functions belonging to the weighted Hilbert space 
\begin{displaymath}
    \mathcal{H}\coloneqq L^2(\mathbb{R}^3,f_\mathcal{M}(\bxi) \mathrm{d} \bxi) = \left\{ f \;\bigg|\; \int_{\mathbb{R}^3} \left( f(\bxi) \right)^2 f_\mathcal{M}(\bxi) \,\mathrm{d} \bxi < + \infty \right\},
\end{displaymath}
which is equipped with the inner product $\langle g_1, g_2 \rangle := \int_{\mathbb{R}^3} g_1(\boldsymbol{\xi})g_2(\boldsymbol{\xi}) f_{\mathcal{M}}(\boldsymbol{\xi}) \mathrm{d}\boldsymbol{\xi}$.


\section{Linearized R13 System} 
\label{Sec:LR13}
The linearized Boltzmann equation for the distribution function $f$ is:
\begin{equation}
    \label{eq:Linearized_Boltzmann}
    \pdfFrac{f}{t}+\sum_{k=1}^{d}\xi_{k}\pdfFrac{f}{\xi_{k}}=\frac{1}{\Kn}\mathcal{L}f,
\end{equation}
where $\mathcal{L}$ is a self-adjoint negative semidefinite linear operator. 
In this section, we outline the spectral and Enskog expansions for the distribution function $f$ and subsequently present the linearized R13 system along with its weak formulation.
\subsection{Moment variables and Chapman–Enskog expansion}
We briefly recall from \cite{lin2025time} the moment basis and the orthogonal decomposition underlying the linearized R13 system. Let
$w_{i_1 \cdots i_l}^n = \langle \psi_{i_1 \cdots i_l}^n, f \rangle,$
where the polynomial basis $\psi_{i_1 \cdots i_l}^n(\boldsymbol{\xi})$ is given by
\begin{equation}
\label{eq:psi}
\psi_{i_1 \cdots i_l}^n(\boldsymbol{\xi}) = \bar{L}_n^{(l+1/2)} \left( \frac{|\boldsymbol{\xi}|^2}{2} \right) \xi_{\langle i_1} \cdots \xi_{i_l \rangle},
\end{equation}
with $\bar{L}_n^{(l+1/2)}$ denoting the normalized associated Laguerre polynomial:
\begin{displaymath}
\bar{L}_n^{(l+1/2)}(x) = \sqrt{\frac{\sqrt{\pi}}{2^{l+1} n! \Gamma (n+l+3/2)}} x^{-(l+1/2)} \left( \frac{\mathrm{d}}{\mathrm{d}x} - 1 \right)^n x^{n+l+1/2}.
\end{displaymath}
These moments are related to Grad's 13 moments through the relations: 
\begin{displaymath}
\rho = w^0, \quad u_i = \sqrt{3} w_i^0, \quad \theta = -\sqrt{\frac{2}{3}} w^1, \quad \sigma_{ij} = \sqrt{15} w_{ij}^0, \quad s_i = -\sqrt{\frac{15}{2}} w_i^1,
\end{displaymath}
where $\rho$, $u_i$, $\theta$, $\sigma_{ij}$, and $s_i$ represent the density, velocity, temperature, shear stress, and heat flux, respectively.


Applying the Chapman–Enskog expansion with $\Kn$ as the small parameter determines the asymptotic order of the moments and the linear dependencies among them, see \cite{lin2025time} for details. This yields the R13 basis $\tilde{\bm{\phi}}_{13}\coloneqq\{\psi^0,\psi_{i}^{0},\psi^1,\phi_{ij}^{0},\phi_{i}^1\}$ and the corresponding variables:
\begin{displaymath}
\rho, \quad \theta, \quad u_i, \quad \bar{s}_i = \langle \phi_i^1, f \rangle, \quad \bar{\sigma}_{ij} = \langle \phi_{ij}^0, f \rangle,
\end{displaymath}
where basis functions $\phi_i^1$ and $\phi_{ij}^0$ are linear combinations of the polynomials $\psi_{i_1\cdots i_l}^n$, 
\begin{equation}
\label{eq:phi}
\phi_i^{1} = \sum_{n=1}^{+\infty} c_1^{1,n} \psi_i^n, \qquad
\phi_{ij}^{0} = \sum_{n=0}^{+\infty} c_2^{0,n} \psi_{ij}^n,
\end{equation}
and
\begin{equation} \label{eq:c_normalization}
    \sum_{n=1}^{+\infty} \left( c_1^{1,n} \right)^2 = \frac{15}{2}, \qquad \sum_{n=0}^{+\infty} \left( c_2^{0,n} \right)^2 = 15.
\end{equation}
These coefficients $c_1^{1,n}$ and $c_2^{0,n}$ depend on the collision model, and they are determined by asymptotic analysis, which guarantees for any function $\phi$ satisfying $\langle \phi, \tilde{\boldsymbol{\phi}}_{13} \rangle = 0$, the inner product $\langle \phi, f\rangle$ has order $o(\Kn)$ if it is finite. The normalization \eqref{eq:c_normalization} is chosen such that $\bar{s}_i$ and $\bar{\sigma}_{ij}$ reduce to the classical heat flux and shear stress for Maxwell molecules.

To characterize the $\mathcal{O}(\Kn^2)$ contribution required at the super-Burnett order, \cite{lin2025time} further constructs a second-order basis $\tilde{\bm{\phi}}{}^{(2)}$, orthogonal to $\tilde{\bm{\phi}}_{13}$, and obtains the orthogonal decomposition: 

\begin{displaymath}
\mathcal{H}= \mathbb{V}^{(0)} \oplus \mathbb{V}^{(1)} \oplus \mathbb{V}^{(2)} \oplus \mathbb{V}^{(\mathrm{r})},
\end{displaymath}
where $f\in\mathbb{V}^{(k)}$ satisfies $f=\mathcal{O}(\Kn^{k})$ for $k=0,1,2$, and $\mathbb{V}^{(r)}$ contains higher-order remainders.

Denoting by $\mathcal{P}^{(k)}$ the projection operator from $\mathcal{H}$ onto $\mathbb{V}^{(k)}$, we write:
\begin{equation}
\label{eq:Enskog_ext}
f = f^{(0)} + f^{(1)} + f^{(2)} + f^{(\mathrm{r})},
\end{equation}
where $f^{(k)} = \mathcal{P}^{(k)} f $ for $k=0,1,2$, and $f^{(\mathrm{r})} = \mathcal{P}^{(\mathrm{r})} f \in \mathbb{V}^{(\mathrm{r})}$ is the high-order remainder. In particular, \cite{lin2025time} gives:
\begin{equation}
    \label{eq:expression_for_f0_f1}
    \begin{aligned}
        f^{(0)}=\rho\psi^{0}-\sqrt{\frac{3}{2}}\theta\psi^{1}+\sum_{i=1}^{3}\sqrt{3}u_i\psi_{i}^{0},\quad f^{(1)}=\frac{2}{5}\sum_{i=1}^3\bar{s}_i\phi_{i}^1+\frac{1}{2}\sum_{i,j=1}^{3}\bar{\sigma}_{ij}\phi_{ij}^{0}.
    \end{aligned}
\end{equation}
Define the projection operators $\mathcal{P}^{(i)} f := f^{(i)}$ and $\mathcal{P}^{(01)}\coloneqq\mathcal{P}^{(0)}+\mathcal{P}^{(1)}$, by letting $f_0 := f^{(0)} + f^{(1)} + f^{(2)}$, the linearized R13 system can be expressed as:
\begin{equation}
    \label{eq:R13_distribution}
    \pdfFrac{}{t}(\mathcal{P}^{(01)}f_0)+\sum_{k=1}^{d}\pdfFrac{}{x_k}(\mathcal{A}_{k}f_0)=\frac{1}{\Kn}\tilde{\mathcal{L}}f_0,
\end{equation}
where the operators $\mathcal{A}_k$ and $\tilde{\mathcal{L}}$ are defined by:
\begin{equation}
    \label{eq:approximated_ops}
    \begin{aligned}
    \mathcal{A}_{k}&\coloneqq\mathcal{P}^{(01)}\xi_{k}\mathcal{P}^{(01)}+\mathcal{P}^{(2)}\xi_{k}\mathcal{P}^{(01)}+\mathcal{P}^{(01)}\xi_{k}\mathcal{P}^{(2)},\\\
    \tilde{\mathcal{L}}&\coloneqq (\mathcal{P}^{(1)}+\mathcal{P}^{(2)})\mathcal{L}(\mathcal{P}^{(1)}+\mathcal{P}^{(2)}).
    \end{aligned}
\end{equation}

\subsection{Moment equations and boundary conditions} For the aforementioned macroscopic variables, the governing R13 equations read:
\begin{equation}\label{eq:R13_evolution}
\begin{cases}
\partial_t \rho = -\nabla \cdot \boldsymbol{u},\quad\partial_t \theta = -\frac{2}{3}\nabla \cdot \boldsymbol{u} - \frac{2}{3}k_0\nabla \cdot \bar{\boldsymbol{s}} + k_1 \mathrm{Kn} \Delta \theta - k_2 \mathrm{Kn} \nabla \cdot (\nabla \cdot \bar{\boldsymbol{\sigma}}), \\
\partial_t \boldsymbol{u} = -\nabla \rho - \nabla \theta + k_3 \mathrm{Kn} \nabla \cdot (\nabla \boldsymbol{u})_{\mathrm{stf}}+ k_4 \mathrm{Kn} \nabla \cdot (\nabla \bar{\boldsymbol{s}})_{\mathrm{stf}} - k_5 \nabla \cdot \bar{\boldsymbol{\sigma}}, \\
\partial_t \bar{\boldsymbol{s}} = -\frac{5}{2}k_0\nabla \theta + \frac{5}{2}k_4 \mathrm{Kn} \nabla \cdot (\nabla \boldsymbol{u})_{\mathrm{stf}} + 2k_6 \mathrm{Kn} \nabla(\nabla \cdot \bar{\boldsymbol{s}})\\
\quad+ \frac{12}{5} k_7 \mathrm{Kn} \nabla \cdot (\nabla \bar{\boldsymbol{s}})_{\mathrm{stf}} - k_8 \nabla \cdot \bar{\boldsymbol{\sigma}} - \frac{2l_1}{3\mathrm{Kn}}\bar{\boldsymbol{s}}, \\
\partial_t \bar{\boldsymbol{\sigma}} = -3k_2 \mathrm{Kn} (\nabla^2 \theta)_{\mathrm{stf}} - 2k_5 (\nabla \boldsymbol{u})_{\mathrm{stf}} - \frac{4}{5} k_8 (\nabla \bar{\boldsymbol{s}})_{\mathrm{stf}} \\
\quad + 2k_9 \mathrm{Kn} \nabla \cdot (\nabla \bar{\boldsymbol{\sigma}})_{\mathrm{stf}} + k_{10} \mathrm{Kn} \big(\nabla (\nabla \cdot \bar{\boldsymbol{\sigma}})\big)_{\mathrm{stf}} - \frac{l_2}{\mathrm{Kn}}\bar{\boldsymbol{\sigma}}.
\end{cases}
\end{equation}
The coefficients $k_i$ and $l_i$ depend on the specific gas molecular model. Following \cite{lin2025time}, \eqref{eq:R13_evolution} is equivalent to \eqref{eq:R13_distribution}, and the auxiliary variables $\bar{\boldsymbol{\sigma}}$ and $\bar{\boldsymbol{s}}$ relate to the physical stress $\boldsymbol{\sigma}$ and heat flux $\boldsymbol{s}$ via:
\begin{equation}\label{eq:qbartoq}
\sigmab = k_5\sigmabarb-k_4 \Kn (\nabla \qbarb)_{\stf} -k_3 \Kn(\nabla \bm{u})_{\stf}, \quad \qb = k_0\qbarb-\frac{3}{2}k_1\Kn \nabla\theta+\frac{3}{2}k_2\Kn\nabla\cdot \sigmabarb.
\end{equation}
Let $(\boldsymbol{t}_1,\boldsymbol{t}_2,\boldsymbol{n})$ be a boundary-aligned orthonormal frame with outer unit normal $\boldsymbol{n}$. We denote the projected components by $u_a := \boldsymbol{u} \cdot \boldsymbol{a}$ and $\sigma_{ab} := \sum_{i,j} \sigma_{ij} a_i b_j$ for $\bm{a}, \bm{b} \in \{\bm{t}_1, \bm{t}_2, \bm{n}\}$. Assuming wall temperature $\theta^W$ and velocity $\boldsymbol{u}^W$, the Onsager boundary conditions take the form:
\begin{equation}
\left\{\begin{array}{@{\!\!\!\!\!}ll}
\label{eq:Onsager_detail}
     &u_n = 0, \\
     &\bar{s}_{n}  = \tilde{\chi} \left[ m_{11} (\theta -\theta^{W}) +m_{12} \bar{\sigma}_{nn} - \Kn \left(m_{13}\frac{\partial \bar{s}_j}{\partial x_j}+m_{14}\frac{\partial \bar{s}_{\langle n}}{\partial x_{n \rangle}}+ m_{15}\frac{\partial u_{\langle n}}{\partial x_{n \rangle}}\right) \right], \\ 
      &m_{26}\bar{s}_{n} + m_{27} \Kn \frac{\partial \theta}{\partial x_n} - m_{28} \Kn \frac{\partial \bar{\sigma}_{nj}}{\partial x_j}   \\
     &\quad =\tilde{\chi} \left[-m_{21} (\theta -\theta^{W}) + m_{22} \bar{\sigma}_{nn} + \Kn  \left( m_{23} \frac{\partial \bar{s}_j}{\partial x_j} +  m_{24} \frac{\partial \bar{s}_{\langle n}}{\partial x_{n \rangle}}  + m_{25} \frac{\partial u_{\langle n}}{\partial x_{n \rangle}} \right) \right],  \\ 
      & \bar{\sigma}_{t_i n} = \tilde{\chi} \left[m_{31} (u_{t_i} -u^{W}_{t_i}) +m_{32}\bar{s}_{t_i} - \Kn\left(m_{33}\frac{\partial \bar{\sigma}_{t_i j}}{\partial x_j}  + m_{34}\frac{\partial \bar{\sigma}_{\langle t_i n}}{\partial x_{n \rangle}}  - m_{35}\frac{\partial \theta}{\partial x_{t_i}}\right) \right],\\ 
       &   m_{46}\bar{\sigma}_{t_in} + m_{47}\Kn\frac{\partial u_{\langle t_i}}{\partial x_{n\rangle}} + m_{48}\Kn\frac{\partial \bar{s}_{\langle t_i}}{\partial x_{n\rangle}}  \\
      & \quad = - \tilde{\chi} \left[-m_{41} (u_{t_i} -u^{W}_{t_i}) +m_{42}\bar{s}_{t_i} +\Kn \left(m_{43} \frac{\partial \bar{\sigma}_{t_i j}}{\partial x_j}  + m_{44} \frac{\partial \bar{\sigma}_{\langle t_i n}}{\partial x_{n \rangle}}  - m_{45} \frac{\partial \theta}{\partial x_{t_i}} \right) \right], \\ 
        &   m_{56}\bar{\sigma}_{t_in} + m_{57}\Kn\frac{\partial u_{\langle t_i}}{\partial x_{n\rangle}} + m_{58}\Kn\frac{\partial \bar{s}_{\langle t_i}}{\partial x_{n\rangle}}   \\
       & \quad  = - \tilde{\chi} \left[-m_{51} (u_{t_i} -u^{W}_{t_i}) +m_{52}\bar{s}_{t_i} +\Kn \left(m_{53}  \frac{\partial \bar{\sigma}_{t_i j}}{\partial x_j}  + m_{54} \frac{\partial \bar{\sigma}_{\langle t_i n}}{\partial x_{n \rangle}}  - m_{55} \frac{\partial \theta}{\partial x_{t_i}} \right) \right],\\ 
  &m_{66} \bar{s}_n + \Kn\left(m_{67} \frac{\partial \bar{\sigma}_{\langle nn }}{\partial x_{n \rangle}} 
 +m_{68}\frac{\partial \bar{\sigma}_{nj }}{\partial x_{j}} + m_{69}\frac{\partial \theta}{\partial x_n} \right) \\
 & \quad = - \tilde{\chi} \left[-m_{61} (\theta -\theta^{W}) +m_{62} \bar{\sigma}_{nn} + \Kn \left( m_{63} \frac{\partial \bar{s}_j}{\partial x_j} +  m_{64} \frac{\partial \bar{s}_{\langle n}}{\partial x_{n \rangle}} + m_{65} \frac{\partial u_{\langle n}}{\partial x_{n \rangle}} \right)\right],\\ 
  &\Kn\left(\frac{\partial \bar{\sigma}_{\langle t_i t_i}}{\partial x_{n \rangle}} + \frac{1}{2} \frac{\partial \bar{\sigma}_{\langle n n}}{\partial x_{n \rangle}} \right) =  -\tilde{\chi} m_{71} \left( \bar{\sigma}_{t_i t_i}  + \frac{1}{2} \bar{\sigma}_{nn}  \right),\quad\Kn \frac{\partial \bar{\sigma}_{\langle t_1 t_2}}{\partial x_{n\rangle}}  = -\tilde{\chi} m_{81} \bar{\sigma}_{t_1 t_2}.
\end{array} \right.
\end{equation}
Here, Einstein summation is implied over the repeated index $j$, and the conditions hold for each tangential direction $i \in \{1, 2\}$. The coefficients $m_{jk}$ parameterize the molecular interactions, and $\tilde{\chi}\coloneqq\frac{2\chi}{2-\chi}>0$ represents the modified Maxwell accommodation factor as introduced in \cite{westerkamp2019finite}. As a canonical example, the explicit expressions of the coefficients $k_i$, $l_i$, and $m_{jk}$ for the inverse $\eta$-th power intermolecular potential were derived in \cite{lin2025time}.

The boundary conditions in \eqref{eq:Onsager_detail} are a discrete form of Maxwell's boundary condition for the kinetic equation, and therefore can also be written as a single equation of $f_0$ \cite{lin2025time}. Such a representation is useful for the derivation of the entropy inequality and will be introduced later in \cref{subsec:Bdry_Term_for_Entropy}.
\begin{remark}
    For Maxwell molecules (i.e., $\eta = 5$), certain coefficients $k_i$ and $m_{ij}$ vanish, reducing the governing equations to the R13 system presented in \cite{lewintan2025wellposedness} (see \cite[Appendix B]{lin2025time} for detailed parameter mappings).
\end{remark}
\section{Entropy inequalities}
\label{Sec:Entropy}
Entropy inequalities, which encode the second law of thermodynamics, play a central role in establishing the well-posedness of the R13 system. While previous work \cite{lin2025time} has successfully derived an entropy inequality for the whole space $\mathbb{R}^{3}$, it did not account for the bounded domain $\Omega$ and the associated Onsager boundary conditions on $\partial\Omega$. Therefore, in this section, we formulate and rigorously establish the entropy inequalities for the R13 system on a bounded domain $\Omega \subset \mathbb{R}^3$.
\subsection{Compact form}
For convenience, we formulate the compact form of the R13 system \eqref{eq:R13_evolution} subject to the Onsager boundary conditions \eqref{eq:Onsager_detail} with wall boundary conditions. 
For each $\mathbf{x}\in\Omega$ and $t\in I$, defining the state vector $\boldsymbol{\mathcal{U}}(\mathbf{x},t)\coloneqq(\rho,\theta,\boldsymbol{u},\boldsymbol{\bar{s}},\bar{\boldsymbol{\sigma}})^{T}|_{(\mathbf{x},t)}\in\mathbf{V} := \mathbb{R} \times \mathbb{R} \times \mathbb{R}^3 \times \mathbb{R}^3 \times \mathbb{R}_{\stf}^{3\times3}$, the R13 system can be expressed as:
\begin{equation}
    \label{eq:time_dependent_compact_R13}
    \left\{
    \begin{aligned}
        M \partial_t \boldsymbol{\mathcal{U}} + \mathcal{A}(\nabla) \boldsymbol{\mathcal{U}} &= \mathcal{D}(\nabla, \nabla) \boldsymbol{\mathcal{U}} + S \boldsymbol{\mathcal{U}},\quad \mathbf{x}\in \Omega;\\
        \mathcal{G}(\boldsymbol{\mathcal{U}})(\mathbf{x},t)&= \mathbf{g}_{\text{ext}}(\mathbf{x},t),\quad\mathbf{x}\in\partial\Omega;\\
        \boldsymbol{\mathcal{U}}(\mathbf{x},0)&=\boldsymbol{\mathcal{U}}_0(\mathbf{x}),\quad\mathbf{x}\in\Omega,
    \end{aligned}
    \right.
\end{equation}
where $\mathcal{G}$ is a linear operator associated with the Onsager boundary conditions. The operators $M$ and $S$ represent the mass and relaxation matrices, given respectively by:
\begin{equation}
    \label{eq:mass_ops}
    M = \begin{bmatrix}
1 & 0 & 0 & 0 & 0 \\
0 & \frac{3}{2} & 0 & 0 & 0 \\
0 & 0 & \mathbf{I}_{3} & 0 & 0 \\
0 & 0 & 0 & \frac{2}{5} \mathbf{I}_{3} & 0 \\
0 & 0 & 0 & 0 & \frac{1}{2} \mathbb{I}_{\stf}
\end{bmatrix},\quad S = -\frac{1}{\Kn} \begin{bmatrix}
0 & 0 & 0 & 0 & 0 \\
0 & 0 & 0 & 0 & 0 \\
0 & 0 & 0 & 0 & 0 \\
0 & 0 & 0 & \frac{4}{15}l_1 \mathbf{I}_{3} & 0 \\
0 & 0 & 0 & 0 & \frac{1}{2}l_2 \mathbb{I}_{\stf}
\end{bmatrix},
\end{equation}
where $\mathbf{I}_{3}$ denotes the identity operator on $\mathbb{R}^3$, and the operator $\mathbb{I}_{\stf}$ is the identity operator on $\mathbb{R}_{\stf}^{3\times 3}$. The first-order derivative operator $\mathcal{A}(\nabla)$ is defined as:
\begin{equation}
    \label{eq:ops_A}
    \mathcal{A}(\nabla) = \begin{bmatrix}
0 & 0 & \nabla\cdot & 0 & 0 \\
0 & 0 & \nabla\cdot & k_0\nabla\cdot & 0 \\
\nabla & \nabla & 0 & 0 & k_5\nabla\cdot \\
0 & k_0\nabla & 0 & 0 & \frac{2}{5}k_8\nabla\cdot \\
0 & 0 & k_5\stf~\nabla & \frac{2}{5}k_8\stf~\nabla & 0
\end{bmatrix},
\end{equation}
and the second order derivative operator $\mathcal{D}(\nabla,\nabla)$ is given by:
\normalsize
\begin{equation}
    \begin{aligned}
    \label{eq:ops_D}
&\mathcal{D}(\nabla, \nabla) = \Kn \begin{bmatrix}
0 & 0 & 0 & 0 & 0 \\
0 & \frac{3}{2}k_1\Delta & 0 & 0 & -\frac{3}{2}k_2\nabla\cdot(\nabla\cdot) \\
0 & 0 & k_3\nabla\cdot(\stf~\nabla) & k_4\nabla\cdot(\stf~\nabla ) & 0 \\
0 & 0 & k_4\nabla\cdot(\stf~\nabla)& \mathcal{D}_{44} & 0 \\
0 & -\frac{3}{2}k_2\stf~\nabla^2 & 0 & 0 & \mathcal{D}_{55}
\end{bmatrix},\\
&\mathcal{D}_{44} = \frac{4}{5}k_6\nabla(\nabla\cdot) + \frac{24}{25}k_7\nabla\cdot(\stf~\nabla),\qquad
\mathcal{D}_{55} = k_9\nabla\cdot(\stf~\nabla)+ \frac{1}{2}k_{10}(\stf~\nabla\nabla\cdot),
\end{aligned}
\end{equation}
\normalsize Furthermore, $\mathbf{g}_{\text{ext}}$ denotes the vector of boundary data corresponding to the boundary distribution $f_W$, and $\boldsymbol{\mathcal{U}}_0$ represents the initial state of the system.
\begin{remark}
    The system formulated in \eqref{eq:time_dependent_compact_R13} takes a form analogous to generalized advection-diffusion equations subject to mixed boundary conditions. It is worth noting that the only Dirichlet boundary condition in our system is the non-penetration condition $u_n = 0$; all others contain normal derivatives.
\end{remark}
\subsection{Entropy estimation}
\label{subsec:entropy}
In this section, we derive a formal energy estimate for sufficiently smooth solutions $\boldsymbol{\mathcal{U}}$ to \eqref{eq:time_dependent_compact_R13}, assuming that all derivatives and boundary traces appearing below are well defined and that the Onsager boundary conditions hold in the classical sense. Write the inner product between $\boldsymbol{\mathcal{U}}_1,\boldsymbol{\mathcal{U}}_2\in\mathbf{V}$ as:
\begin{equation}
    \label{eq:inner_prod_macro}
    \innerprod{\boldsymbol{\mathcal{U}}_1}{\boldsymbol{\mathcal{U}}_2}_{\mathbf{V}} := \rho_1 \rho_2 + \theta_1 \theta_2 + \boldsymbol{u}_1 \cdot \boldsymbol{u}_2 + \overline{\boldsymbol{s}}_1 \cdot \overline{\boldsymbol{s}}_2 + \overline{\boldsymbol{\sigma}}_1 : \overline{\boldsymbol{\sigma}}_2,
\end{equation}
By taking the inner product of \eqref{eq:time_dependent_compact_R13} with $\boldsymbol{\mathcal{U}}$ and integrating over the domain $\Omega$, we obtain:
\begin{equation}
    \label{eq:energy_integral}
    \frac{1}{2}\pdfFrac{}{t}\int_{\Omega}\innerprod{\boldsymbol{\mathcal{U}}}{M\boldsymbol{\mathcal{U}}}_{\mathbf{V}}\dif\mathbf{x}+\underbrace{\int_{\Omega}\innerprod{\boldsymbol{\mathcal{U}}}{\mathcal{A}(\nabla)\boldsymbol{\mathcal{U}}}_{\mathbf{V}}\dif\mathbf{x}}_{I_1}=\underbrace{\int_{\Omega}\innerprod{\boldsymbol{\mathcal{U}}}{\mathcal{D}(\nabla,\nabla)\boldsymbol{\mathcal{U}}+S\boldsymbol{\mathcal{U}}}_{\mathbf{V}}\dif\mathbf{x}}_{I_2}.
\end{equation}
Applying the divergence theorem and utilizing the non-penetration condition $u_n = 0$, the term $I_1$ can be evaluated as:
\begin{equation}
    \label{eq:I1}
    \begin{aligned}
        I_1&=\int_{\Omega}\nabla\cdot(\rho\boldsymbol{u}+k_0\theta\bar{\boldsymbol{s}}+\theta\boldsymbol{u}+k_5\bar{\boldsymbol{\sigma}}\boldsymbol{u}+\frac{2}{5}k_8\bar{\boldsymbol{\sigma}}\bar{\boldsymbol{s}})\dif\mathbf{x}\\
        &=\int_{\partial\Omega}\left(k_0\theta s_n+k_5\sum_{i=1}^2\bar{\sigma}_{nt_i}u_{t_i}+\frac{2}{5}k_8\left(\sum_{i=1}^2\bar{\sigma}_{nt_i}\bar{s}_{t_i}+\bar{\sigma}_{nn}\bar{s}_n\right)\right)\dif s\coloneqq \mathcal{F}_1(\boldsymbol{\mathcal{U}}).
    \end{aligned}
\end{equation}
Similarly, using integration by parts, the term $I_2$ can be decomposed into a volume integral $\mathcal{W}_1$ and a boundary integral $\mathcal{F}_2$, such that $I_2 = \mathcal{W}_1 + \mathcal{F}_2$, where:
\begin{equation}
    \label{eq:W1_term}
    \begin{aligned}
    \mathcal{W}_1
    &=-\mathrm{Kn} \int_{\Omega} \left[ \left(\frac{3}{2}k_{1} |\nabla\theta|^2 - 3k_{2} \nabla\theta \cdot (\nabla \cdot \bar{\boldsymbol{\sigma}})+ \frac{1}{2}k_{10} |\nabla \cdot \bar{\boldsymbol{\sigma}}|^2 \right)+ \frac{4}{5}k_{6} |\nabla \cdot \bar{\boldsymbol{s}}|^2 \right.\\
    &
    \left.+\left(k_{3} |(\nabla \boldsymbol{u})_{\stf}|^2 + 2k_{4} (\nabla \boldsymbol{u})_{\stf} : (\nabla \bar{\boldsymbol{s}})_{\stf} + \frac{24}{25}k_{7} |(\nabla \bar{\boldsymbol{s}})_{\stf}|^2\right) + k_{9} |(\nabla \bar{\boldsymbol{\sigma}})_{\stf}|^2  \right] \dif\mathbf{x}\\
    & -\frac{1}{\mathrm{Kn}} \int_{\Omega} \left( \frac{4}{15}l_{1} |\bar{\boldsymbol{s}}|^2 + \frac{1}{2}l_{2} |\bar{\boldsymbol{\sigma}}|^2 \right) \dif \mathbf{x},\\
    \end{aligned}
\end{equation}
\begin{equation}
    \label{eq:F2_term}
    \begin{aligned}
    \mathcal{F}_2&=\mathrm{Kn} \int_{\partial\Omega} \left[ \frac{3}{2}k_{1} \theta \frac{\partial \theta}{\partial n} - \frac{3}{2}k_{2} \theta \boldsymbol{n} \cdot (\nabla \cdot \bar{\boldsymbol{\sigma}}) - \frac{3}{2}k_{2} (\bar{\boldsymbol{\sigma}} \boldsymbol{n}) \cdot \nabla \theta + k_{3} \boldsymbol{u} \cdot ((\nabla \boldsymbol{u})_{\stf} \boldsymbol{n}) \right.\\
    &+ k_{4} \boldsymbol{u} \cdot ((\nabla \bar{\boldsymbol{s}})_{\stf} \boldsymbol{n}) + k_{4} \bar{\boldsymbol{s}} \cdot ((\nabla \boldsymbol{u})_{\stf} \boldsymbol{n}) + \frac{4}{5}k_{6} (\bar{\boldsymbol{s}} \cdot \boldsymbol{n})(\nabla \cdot \bar{\boldsymbol{s}}) + \frac{24}{25}k_{7} \bar{\boldsymbol{s}} \cdot ((\nabla \bar{\boldsymbol{s}})_{\stf} \boldsymbol{n}) \\
    &+ \left.k_{9} \bar{\boldsymbol{\sigma}} : ((\nabla \bar{\boldsymbol{\sigma}})_{\stf} \boldsymbol{n}) + \frac{1}{2}k_{10} (\bar{\boldsymbol{\sigma}} \boldsymbol{n}) \cdot (\nabla \cdot \bar{\boldsymbol{\sigma}}) \right] \dif s.
    \end{aligned}
\end{equation}
According to the parameter constraints established in \cite{lin2025time} (i.e. $k_{i}\ge 0$, $l_j>0$, $3k_2^2\le k_1k_{10}$ and $25k_4^2\le 24k_3k_7$), we have $\mathcal{W}_1 \le 0$. Consequently, the entropy equation on the bounded domain can be written as:
\begin{equation}
    \label{eq:entropy_ineq_1}
    \frac{1}{2}\pdfFrac{}{t}\int_{\Omega}\innerprod{\boldsymbol{\mathcal{U}}}{M\boldsymbol{\mathcal{U}}}_{\mathbf{V}}\dif\mathbf{x}=\mathcal{W}_1+ (\mathcal{F}_2(\boldsymbol{\mathcal{U}})-\mathcal{F}_{1}(\boldsymbol{\mathcal{U}}))\coloneqq \mathcal{W}_1-I_{\mathrm{bdry}}.
\end{equation}
\begin{remark}
Estimating the boundary term $I_{\mathrm{bdry}}$ directly using the moment-based formulation of the Onsager boundary conditions \eqref{eq:Onsager_detail} is technically cumbersome. To overcome this difficulty, we exploit the equivalent alternative formulation \eqref{eq:OnsagerBC_Distribution} to bound this term effectively.
\end{remark}
\subsection{Boundary terms}
\label{subsec:Bdry_Term_for_Entropy}

Since direct substitution of the macroscopic Onsager boundary conditions yields intricate algebraic expressions, we bypass this difficulty by returning to the microscopic perspective. Recall that the R13 equations \eqref{eq:R13_evolution} is equivalent to  \eqref{eq:R13_distribution} formulated for $f_0 = f^{(0)} + f^{(1)} + f^{(2)}$, a function in the finite dimensional space $\mathbb{V}^{(0)} \oplus \mathbb{V}^{(1)} \oplus \mathbb{V}^{(2)}$.
As a result, all functionals of $\boldsymbol{\mathcal{U}}$ can be written equivalently as functionals of $f_0$.
We are then able to rewrite  \eqref{eq:entropy_ineq_1} as an inequality involving $f_0$ instead of $\boldsymbol{\mathcal{U}}$, where one can obtain an alternative formula for $I_{\mathrm{bdry}}$ with a clearer structure, allowing us to find its lower bound more systematically.

Recall that $\mathcal{P}^{(01)}$ denotes the projection operator onto $\mathbb{V}^{(0)} \oplus \mathbb{V}^{(1)}$. We define
\begin{displaymath}
\tilde{f} := f^{(0)} + f^{(1)} = \mathcal{P}^{(01)}f_0.
\end{displaymath}
The following lemma gives a representation for the left-hand side of  \eqref{eq:entropy_ineq_1}:
\begin{lemma}
    \label{Lem:energy_equivalent}
    Let $\boldsymbol{\mathcal{U}}$ be the state vector associated with the truncated distribution function $\tilde{f} \in \mathbb{V}^{(0)} \oplus \mathbb{V}^{(1)}$. It holds that
    \begin{equation}
        \innerprod{\tilde{f}}{\tilde{f}}=\innerprod{\boldsymbol{\mathcal{U}}}{M\boldsymbol{\mathcal{U}}}_{\mathbf{V}}
    \end{equation}
\end{lemma}
\begin{proof}
    By \eqref{eq:expression_for_f0_f1}, adopting the Einstein summation convention over repeated spatial indices, the truncated distribution function $\tilde{f}$ is:
    \begin{equation}
        \label{eq:expression_for_tildef}
        \tilde{f}=\rho\psi^{0}-\sqrt{\frac{3}{2}}\theta\psi^{1}+\sqrt{3}u_i\psi_{i}^{0}+\frac{2}{5}\bar{s}_i\phi_{i}^{1}+\frac{1}{2}\bar{\sigma}_{ij}\phi_{ij}^{0}.
    \end{equation}
By orthogonality, the cross terms between different blocks vanish. Moreover, 
\begin{equation}
    \label{eq:inner_product_property_of_psi}
    \begin{aligned}
    &\innerprod{\psi^{n}}{\psi^{m}}=\delta_{nm},\quad\innerprod{\psi_{i}^{n}}{\psi_{j}^{m}}=\frac{1}{3}\delta_{ij}\delta_{nm},\\
    &\innerprod{\psi_{ij}^{n}}{\psi_{kl}^{m}}=\frac{1}{15}\delta_{nm}\left(\delta_{ik}\delta_{jl}+\delta_{il}\delta_{jk}-\frac{2}{3}\delta_{ij}\delta_{kl}\right).
    \end{aligned}
\end{equation}
Using \eqref{eq:inner_product_property_of_psi} and \eqref{eq:phi}, we obtain:
\begin{equation}
    \label{eq:innerprod_of_s_and_sigma}
    3u_iu_j\innerprod{\psi_{i}^0}{\psi_{j}^0}=\|\bbm{u}\|^2,\quad
    \innerprod{\bar{s}_i\phi_{i}^1}{\bar{s}_j\phi_{j}^1}=\frac{5}{2}\|\bbm{s}\|^2,\quad \innerprod{\frac{1}{2}\bar{\sigma}_{ij}\phi_{ij}^0}{\frac{1}{2}\bar{\sigma}_{kl}\phi_{kl}^0}=\frac{1}{2}\|\bbm{\sigma}\|^2.
\end{equation}
Substituting \eqref{eq:inner_product_property_of_psi} and \eqref{eq:innerprod_of_s_and_sigma} into \eqref{eq:expression_for_tildef} immediately yields:
\begin{equation}
    \label{eq:1st_simplify}
    \begin{aligned}
    \langle \tilde{f}, \tilde{f} \rangle = \rho^2 + \frac{3}{2}\theta^2 + \|\boldsymbol{u}\|^2 + \frac{2}{5}\|\bar{\boldsymbol{s}}\|^2 + \frac{1}{2}\|\bar{\boldsymbol{\sigma}}\|^2 = \langle \bcal{U}, M\bcal{U} \rangle_{\mathbf{V}}.
\end{aligned}
\end{equation}
Then we complete the proof.
\end{proof}

The orthogonality condition $\langle f^{(2)}, \tilde{f} \rangle = 0$ implies $\langle f_0, \tilde{f} \rangle = \langle \tilde{f}, \tilde{f} \rangle$. Testing equation \eqref{eq:R13_distribution} against $f_0$ and integrating over the domain $\Omega$ yields:
\begin{equation}
\label{eq:energy_for_pdf}
\frac{1}{2}\frac{\partial}{\partial t}\int_\Omega \langle \tilde{f}, \tilde{f} \rangle \,\mathrm{d}\mathbf{x} + \sum_{k=1}^d \int_\Omega \left\langle f_0, \frac{\partial}{\partial x_{k}}\mathcal{A}_k f_0 \right\rangle \,\mathrm{d}\mathbf{x} = \frac{1}{\mathrm{Kn}}\int_\Omega \langle f_0, \tilde{\mathcal{L}} f_0 \rangle \,\mathrm{d}\mathbf{x}. 
\end{equation}
In light of \cref{Lem:energy_equivalent}, the energy equalities \eqref{eq:energy_for_pdf} and \eqref{eq:entropy_ineq_1} are identical. An application of the divergence theorem then extracts the kinetic representation of the boundary flux:
\begin{equation}
\label{eq:boundary_value_from_pdf}
I_{\mathrm{bdry}} = \int_{\partial\Omega} \sum_{k=1}^d \left\langle f_0, (\mathcal{A}_k n_k) f_0 \right\rangle \,\mathrm{d}s \coloneqq\int_{\partial\Omega} \langle f_0, \mathcal{A}_n f_0 \rangle \,\mathrm{d}s.
\end{equation}
Below, we will provide a lower bound for this representation of $I_{\mathrm{bdry}}$.

To this end, we need a kinetic representation of the boundary conditions \eqref{eq:Onsager_detail}, expressed in terms of the distribution function $f_0$. The rest of this paragraph is based on the derivation of boundary conditions for R13 equations, presented in \cite{lin2025time}. As boundary conditions hold point-wise for all $\mathbf{x} \in \partial\Omega$ and $t \in I$, we suppress the spatial and temporal dependencies for brevity. The Maxwell's boundary conditions for the linearized R13 equations \eqref{eq:R13_distribution} can be formulated as:
\begin{equation}
    \label{eq:Maxwell_BC}
    f_0(\xi_n,\xi_{t_1},\xi_{t_2})=\chi f_{W}(\xi_{n},\xi_{t_1},\xi_{t_2})+(1-\chi)f_0(-\xi_n,\xi_{t_1},\xi_{t_2}),
\end{equation}
and we define the \textit{wall Maxwellian} as:
\begin{equation}
    \label{eq:wall_Maxwell}
    f_{W}=\rho_{W}+u_{W,t_1}\xi_{t_1}+u_{W,t_2}\xi_{t_2}+\frac{1}{2}\theta_{W}(\|\xi\|^2-3),
\end{equation}
where the density $\rho_W$ is uniquely chosen to ensure $u_n=\innerprod{\xi_{n}}{f_0}=0$ under the non-penetration boundary condition. 
This yields:
\begin{equation}
    \label{eq:rhoW}
    \rho_{W}=\sqrt{2\pi}\innerprod{(\xi_{n})_{+}}{f_0}-\frac{\theta_{W}}{2}.
\end{equation}
Following \cite{cai2024linear}, the Onsager boundary condition can be formulated as:
\begin{equation}
    \label{eq:OnsagerBC_Distribution}
    \mathcal{P}_{\odd} f_0=Q\mathcal{A}_{n}(f_{W}-\mathcal{P}_{\even}f_0).
\end{equation}
where $Q$ is a self-adjoint, negative semi-definite operator on $\mathbb{V}_{\odd}\coloneqq\mathcal{P}_{\mathrm{odd}}(\mathbb{V}_0 \oplus \mathbb{V}_1 \oplus \mathbb{V}_2)$ (i.e., $\langle f, Qf \rangle \le 0$ and $\langle f, Qg \rangle = \langle Qf, g \rangle$ for all $f, g \in \mathbb{V}_{\odd}$). 
The parity operators are defined accordingly:
\begin{equation}
    \label{eq:Projection_to_odd_and_even}
    \begin{aligned}
        \mathcal{P}_{\odd} f(\xi_{n},\xi_{t_1},\xi_{t_2})&\coloneqq\frac{f(\xi_{n},\xi_{t_1},\xi_{t_2})-f(-\xi_n,\xi_{t_1},\xi_{t_2})}{2},\\
        \mathcal{P}_{\even}f(\xi_n,\xi_{t_1},\xi_{t_2})&\coloneqq\frac{f(\xi_n,\xi_{t_1},\xi_{t_2})+f(-\xi_{n},\xi_{t_1},\xi_{t_2})}{2}.
    \end{aligned}
\end{equation}
We note that the constraint $u_n = 0$ implies $Q$ is not surjective on $\mathcal{H}_{\text{odd}} := \mathcal{P}_{\odd}\mathcal{H}$. 
To highlight the boundary condition $u_n=0$, we introduce the modified operator $\tilde{Q}$:
\begin{equation}
    \label{eq:tildeQ}
    \tilde{Q} f\coloneqq
    \left\{
    \begin{aligned}
    &Qf,\quad Q\xi_{n}=0;\\
    &Qf-\frac{\innerprod{\xi_n}{Qf}}{\innerprod{\xi_n}{Q\xi_n}}Q\xi_n,\quad Q\xi_{n}\neq 0,
    \end{aligned}
    \right.
\end{equation}
which ensures $\xi_{n}\in\ker\tilde{Q}$. 
\begin{remark}
    The modified operator $\tilde{Q}$ is well-defined. Since $Q$ is self-adjoint and negative semi-definite, the generalized Cauchy-Schwarz inequality implies $\forall f\in\mathcal{H}$, $|\langle Q\xi_n, f \rangle|^2 \le \langle \xi_n, Q\xi_n \rangle \langle f, Qf \rangle$. Thus, $\langle \xi_n, Q\xi_n \rangle = 0$ strictly necessitates $Q\xi_n = 0$, precluding any division by zero in the second branch of \eqref{eq:tildeQ}.
\end{remark}

Introducing the simplified notation $f_{\odd} := \mathcal{P}_{\odd}f_0$ and $f_{\even} := \mathcal{P}_{\even}f_0$, we have the following lemma showing that replacing $Q$ with $\tilde{Q}$ in  \eqref{eq:OnsagerBC_Distribution} makes no difference in the boundary condition:
\begin{lemma}
    \label{Lem:equivalence_of_ops}
    Assume the Onsager boundary condition \eqref{eq:OnsagerBC_Distribution} holds and $u_n=0$ on $\partial\Omega$, then $\forall f_0\in\mathcal{H}$, 
    \begin{equation}
        \label{eq:equivalence_of_ops}
        Q\mathcal{A}_n(f_{W}-f_{\even})=\tilde{Q}\mathcal{A}_n(f_{W}-f_{\even}).
    \end{equation}
\end{lemma}
\begin{proof}
   It suffices to consider the non-trivial case where $Q\xi_{n}\neq 0$. The non-penetration condition $u_n=0$ yields $\innerprod{\xi_{n}}{f}=0$, which straightforwardly implies $\innerprod{\xi_n}{f_{\odd}}=0$. Applying the Onsager boundary condition \eqref{eq:OnsagerBC_Distribution} yields:
    \begin{equation}
        \label{eq:Orthogonal}
        \innerprod{\xi_n}{Q\mathcal{A}_n(f_{W}-f_{\even})}=0.
    \end{equation}
    Substituting \eqref{eq:Orthogonal} back into the definition of $\tilde{Q}$ in \eqref{eq:tildeQ} immediately yields \eqref{eq:equivalence_of_ops}.
\end{proof}

In light of \cref{Lem:equivalence_of_ops}, the Onsager boundary condition can be equivalently recast using the modified operator as:
\begin{equation}
\label{eq:equivform_OnsagerBC_Distribution}
\mathcal{P}_{\odd} f=\tilde{Q}\mathcal{A}_{n}(f_{W}-\mathcal{P}_{\even}f).
\end{equation}
From \cite[Section 5]{lin2025time}, \eqref{eq:equivform_OnsagerBC_Distribution} is equivalent to its detailed expression \eqref{eq:Onsager_detail}. 
Recall from \eqref{eq:approximated_ops} that the operator $\mathcal{A}_n$ maps even distribution functions to odd ones, and vice versa.
By definition, $f_{\odd}$ and $f_{\even}$ are orthogonal, implying $\innerprod{f_{\odd}}{f_{\even}}=0$. Introducing the boundary inner product $(f,g)_{b}\coloneqq\int_{\partial\Omega}\innerprod{f}{g}\dif s$, it follows from \eqref{eq:boundary_value_from_pdf} and the modified Onsager boundary condition \eqref{eq:equivform_OnsagerBC_Distribution} that:
\begin{equation}
    \label{eq:energy}
    \begin{aligned}
    I_{\mathrm{bdry}}&=(f_0,\mathcal{A}_nf_0)_b=(f_{\odd}+f_{\even},\mathcal{A}_{n}f_{\even}+\mathcal{A}_n f_{\odd})_b\\
    &=(f_{\odd},\mathcal{A}_n f_{\even})_b+(f_{\even},\mathcal{A}_n f_{\odd})=2(f_{\odd},\mathcal{A}_n f_{\even})_b\\
    &=2(\tilde{Q}\mathcal{A}_n(f_{W}-f_{\even}),\mathcal{A}_nf_{\even})_b\\
    &=2(\tilde{Q}\mathcal{A}_nf_{W},\mathcal{A}_nf_{\even})_b-2(\tilde{Q}\mathcal{A}_n f_{\even},\mathcal{A}_n f_{\even})_b.
    \end{aligned}
\end{equation}
Since $\rho_W$ is a zeroth-order moment, it holds that $\mathcal{A}_{n}\rho_{W}=\rho_{W}\xi_{n}$, leading to $\tilde{Q}\mathcal{A}_n\rho_W=0$. Denoting $\tilde{f}_{W}\coloneqq f_{W}-\rho_{W}$, which depends only on the wall properties, one has
\begin{equation}
    \label{eq:inner_product_of_tildef}
    (\tilde{Q}\mathcal{A}_nf_{W},\mathcal{A}_nf_{\even})_b=(\tilde{Q}\mathcal{A}_{n}\tilde{f}_{W},\mathcal{A}_{n}f_{\even})_b.
\end{equation}
Since $Q$ is negative semi-definite, its modification $\tilde{Q}$, which acts as a Schur complement projecting out the $\xi_n$ direction, rigorously preserves the negative semi-definiteness. From the generalized Cauchy-Schwarz inequality, we can bound the boundary term as follows:
\begin{equation}
    \label{eq:boundary_estimate}
    \begin{aligned}
    I_{\mathrm{bdry}}&\ge(\tilde{Q}\mathcal{A}_n\tilde{f}_{W},\mathcal{A}_n\tilde{f}_{W})_b+(\tilde{Q}\mathcal{A}_n f_{\even},\mathcal{A}_n f_{\even})_b-2(\tilde{Q}\mathcal{A}_n f_{\even},\mathcal{A}_n f_{\even})_b\\
    &\ge (\tilde{Q}\mathcal{A}_n\tilde{f}_{W},\mathcal{A}_n\tilde{f}_{W})_b-(\tilde{Q}\mathcal{A}_n f_{\even},\mathcal{A}_n f_{\even})_b\ge(\tilde{Q}\mathcal{A}_{n}\tilde{f}_{W},\mathcal{A}_{n}\tilde{f}_{W})_b.
    \end{aligned}
\end{equation}
Since $\tilde{f}_W$ is fully determined by $\boldsymbol{u}_W$ and $\theta_W$, the right hand side of \eqref{eq:boundary_estimate} constitutes a negative semi-definite quadratic form in terms of $\mathbf{g}_{\text{ext}}$. Consequently, there exists a negative semi-definite matrix $M_b(\mathbf{x})$ such that:
\begin{equation}
    \label{eq:boundary_by_RHS}
    I_{\mathrm{bdry}}\ge\int_{\partial\Omega}\mathbf{g}_{\ext}^{T}M_b(\mathbf{x})\mathbf{g}_{\ext}\dif s.
\end{equation}
Ultimately, this leads us to the main entropy inequality for the bounded domain $\Omega$: 
\begin{theorem}
    \label{Thm:entropy_inequality}
Under the constraints $3k_2^2 \le k_1k_{10}$ and $25k_4^2 \le 24k_3k_7$, the entropy inequality on a bounded domain $\Omega$ yields:
\begin{equation}
    \label{eq:entropy_ineq_final}
    \frac{1}{2}\pdfFrac{}{t}\int_{\Omega}\innerprod{\boldsymbol{\mathcal{U}}}{M\boldsymbol{\mathcal{U}}}_{\mathbf{V}}\dif\mathbf{x}\le \mathcal{W}_1-\int_{\partial\Omega}\mathbf{g}_{\mathrm{ext}}^{T}M_b(\mathbf{x})\mathbf{g}_{\mathrm{ext}}\dif s,
\end{equation}
where $\mathcal{W}_1\le 0$ is the bulk entropy, detailed in \eqref{eq:W1_term}.
\end{theorem}
\begin{proof}
    The result follows immediately from combining \eqref{eq:entropy_ineq_1} with the boundary estimate \eqref{eq:boundary_by_RHS}.
\end{proof}
\begin{remark}[H-theorem]
    The constraints $3k_2^2\le k_1k_{10}$ and 
    $25k_{4}^2\le 24k_3k_7$, first derived in \cite[Section 3.3]{lin2025time}, ensure that the bulk dissipation term $\mathcal{W}_1$ is non-positive. Consequently, the R13 system is thermodynamically consistent and complies with the second law of thermodynamics; see also the discussion of the H-theorem in \cite{ottinger2010thermodynamically}.
    Let us define the \textit{entropy density} $H$ as
    \begin{equation}
        \label{eq:entropy_density}
        H\coloneqq H_{0}-\frac{1}{2}\left(\rho^2+\frac{3}{2}\theta^2+\|\boldsymbol{u}\|^2+\frac{2}{5}\|\bar{\boldsymbol{s}}\|^2+\frac{1}{2}\|\bar{\boldsymbol{\sigma}}\|^2\right),
    \end{equation}
    where $H_0$ is an arbitrary constant. Under homogeneous boundary conditions (i.e. $\mathbf{g}_{\ext}=\mathbf{0}$), \cref{Thm:entropy_inequality} guarantees that:
    \begin{equation}
        \label{eq:second_law_of_thermodynamics}
        \frac{\dif}{\dif t}\int_{\Omega}H\dif\mathbf{x}\ge 0,
    \end{equation}
    which is consistent with the second law of thermodynamics.
\end{remark}

\section{Well-posedness of steady state problems}
\label{Sec:Steady_State}
In this section, we study the \textit{steady-state problem} associated with the R13 system, corresponding to the stationary regime of \eqref{eq:time_dependent_compact_R13} in which $\frac{\partial \bcal{U}}{\partial t} = 0$. Our main objective is to identify an appropriate variational framework and establish well-posedness via the LBB theorem \cite{boffi2013mixed}. To this end, we reformulate the stationary R13 system as mixed saddle-point problems tailored to the Maxwell and non-Maxwell molecules.



\subsection{Mixed formulations}
\label{subsec:Mixed_Form}
To study the steady-state problem by the LBB framework, we first rewrite the stationary R13 system in mixed form. Since the coercivity mechanism differs substantially between non-Maxwell and Maxwell molecules, the two regimes require different saddle-point formulations.

For steady-state problems, it is convenient to introduce the pressure variable 
$p\coloneqq\rho + \theta$ to replace the density $\rho$, which reduces the number of coupling terms in the weak form. We temporarily denote the corresponding function spaces by $\mathrm{V}_p, \mathrm{V}_\theta, \mathrm{V}_u, \mathrm{V}_s$, and $\mathrm{V}_\sigma$, their precise definitions will be specified in the subsequent subsections when the two regimes are treated separately.

To facilitate the analysis, weak forms of R13 equations will be written in different forms for Maxwell and non-Maxwell molecules. For non-Maxwell molecules ($\eta \neq 5$), we define the composite space $\mathrm{T}_1 := \mathrm{V}_{\bm{s}} \times \mathrm{V}_{\bm{u}} \times \mathrm{V}_{\bm{\sigma}} \times \mathrm{V}_{\theta}$ and partition the variables as:
\begin{equation}
\label{eq:grouping1}
\begin{cases}
\text{Unknowns:}\quad
\boldsymbol{\mathcal{S}} \coloneqq (\bar{\boldsymbol{s}}, \boldsymbol{u}, \bar{\bm{\sigma}},\theta) \in \mathrm{T}_1, 
\quad \bcal{W}\coloneqq p \in \mathrm{V}_p, \\
\text{Test Functions:}\quad
\boldsymbol{\mathcal{R}}
\coloneqq (\bar{\boldsymbol{r}},
\boldsymbol{v},\bar{\bm{\tau}},\gamma) \in \mathrm{T}_1, \quad \bcal{V}\coloneqq q\in\mathrm{V}_p.
\end{cases}
\end{equation}
The corresponding saddle-point problem reads: given $\mathscr{F}_1 \in \mathrm{T}_1^*$, find $\mathcal{S} \in \mathrm{T}_1$ and $p \in V_p$ such that
\begin{equation}
\label{eq:weak_for_nM}
\begin{aligned}
  \mathscr{A}_1(\boldsymbol{\mathcal{S}},\boldsymbol{\mathcal{R}})+ \mathscr{B}_1(\boldsymbol{\mathcal{R}},p) &= \mathscr{F}_1(\boldsymbol{\mathcal{R}}) \quad &\forall\, \boldsymbol{\mathcal{R}} \in \mathrm{T}_1,
  \\
  \mathscr{B}_1(\boldsymbol{\mathcal{S}},q) &= 0 \quad &\forall\, q \in \mathrm{V}_p,
\end{aligned}   
\end{equation}
where the associated functionals are:
\begin{equation}
    \label{eq:expression_for_operators_1}
    \begin{cases}
        \mathscr{A}_{1}(\bcal{S},\bcal{R})=[a(\bar{\bm{s}},\bar{\bm{r}})+j(\bbm{r},\bm{u})+j(\bbm{s},\bm{v})+f(\bm{u},\bm{v})]-c(\bbm{r},\bbm{\sigma})-b(\theta,\bbm{r})+e(\bm{v},\bbm{\sigma})\\
        \phantom{\mathscr{A}_{1}(\bcal{S},\bcal{R})} +[d(\bbm{\sigma},\bbm{\tau})+z(\gamma,\bbm{\sigma})+z(\theta,\bbm{\tau})+h(\theta,\gamma)]+c(\bbm{s},\bbm{\tau})+b(\gamma,\bbm{s})-e(\bm{u},\bbm{\tau}),\\
        \mathscr{B}_1(\bcal{S},q)=-g(q,\bm{u}),\quad \mathscr{F}_1(\bcal{R})=L_1(\bbm{r})+L_4(\bm{v})-L_3(\bbm{\tau})-L_2(\gamma),
    \end{cases}
\end{equation}
Adopting the Einstein summation convention over the tangential index $i \in \{1, 2\}$, the exact expressions for these underlying forms and functionals are: 
\begingroup
\small
\begin{equation}
\label{eq:detailed_bilinear_forms}
\begin{cases}
\begin{aligned}
a(\boldsymbol{s},\boldsymbol{r})\coloneqq&
\frac{24}{25}k_7\Kn(\mathrm{stf}(\nabla\boldsymbol{s}),\mathrm{stf}(\nabla\boldsymbol{r}))
+\frac45 k_6\Kn(\nabla\!\cdot\!\boldsymbol{s},\nabla\!\cdot\!\boldsymbol{r})\\
&+S_5(s_n,r_n)_b+S_1(s_{t_i},r_{t_i})_b
+\frac{4l_1}{15\Kn}(\boldsymbol{s},\boldsymbol{r}),
\end{aligned}
\\
\begin{alignedat}{2}
b(\theta,\boldsymbol{r})\coloneqq&
k_0(\theta,\nabla\!\cdot\!\boldsymbol{r})+R_1(\theta,r_n)_b,\\
c(\boldsymbol{r},\boldsymbol{\sigma})\coloneqq&
-\frac25 k_8(\boldsymbol{r},\nabla\!\cdot\!\boldsymbol{\sigma})
+R_2(r_{t_i},\sigma_{nt_i})_b+R_3(r_n,\sigma_{nn})_b,
\end{alignedat}
\\[0.3em]
\begin{aligned}
d(\boldsymbol{\sigma},\boldsymbol{\tau})\coloneqq&
k_9\Kn(\mathrm{stf}(\nabla\boldsymbol{\sigma}),\mathrm{stf}(\nabla\boldsymbol{\tau}))
+\frac12 k_{10}\Kn(\nabla\!\cdot\!\boldsymbol{\sigma},\nabla\!\cdot\!\boldsymbol{\tau})
+\frac{l_2}{2\Kn}(\boldsymbol{\sigma},\boldsymbol{\tau})
+S_3(\sigma_{nn},\tau_{nn})_b\\
+&S_6(\sigma_{t_1t_1}+\tfrac12\sigma_{nn},
      \tau_{t_1t_1}+\tfrac12\tau_{nn})_b
+S_7(\sigma_{nt_i},\tau_{nt_i})_b
+S_8(\sigma_{t_1t_2},\tau_{t_1t_2})_b,
\end{aligned}
\\[0.3em]
\begin{alignedat}{2}
e(\boldsymbol{v},\boldsymbol{\sigma})\coloneqq&
k_5(\boldsymbol{v},\nabla\!\cdot\!\boldsymbol{\sigma})
+R_4(\sigma_{nt_i},v_{t_i})_b,
\;&
f(\boldsymbol{u},\boldsymbol{v})\coloneqq&
k_3\Kn(\mathrm{stf}(\nabla\boldsymbol{u}),\mathrm{stf}(\nabla\boldsymbol{v}))
+S_2(u_{t_i},v_{t_i})_b,
\end{alignedat}
\\[0.3em]
\begin{alignedat}{2}
g(p,\boldsymbol{v})\coloneqq&
(p,\nabla\!\cdot\!\boldsymbol{v}),
\quad&j(\boldsymbol{s},\boldsymbol{v})\coloneqq&
k_4\Kn(\mathrm{stf}(\nabla\boldsymbol{s}),\mathrm{stf}(\nabla\boldsymbol{v}))
+T_1(s_{t_i},v_{t_i})_b,
\end{alignedat}
\\[0.3em]
\begin{alignedat}{2}
z(\theta,\boldsymbol{\tau})\coloneqq&
-\frac32 k_2\Kn(\nabla\theta,\nabla\!\cdot\!\boldsymbol{\tau})
+T_2(\theta,\tau_{nn})_b,
\;&
h(\theta,\gamma)\coloneqq&
\frac32 k_1\Kn(\nabla\theta,\nabla\gamma)+S_4(\theta,\gamma)_b,
\end{alignedat}
\\[0.3em]
\begin{alignedat}{2}
L_1(\boldsymbol{r})\coloneqq&
-\frac25 C_1m_{11}(\theta^W,r_n)_b
+T_1(u^W_{t_i},r_{t_i})_b,
\qquad&
L_2(\gamma)\coloneqq&
S_4(\theta^W,\gamma)_b,
\end{alignedat}
\\[0.3em]
\begin{alignedat}{2}
L_3(\boldsymbol{\tau})\coloneqq&
T_2(\theta^W,\tau_{nn})_b
-(R_4+k_5)(u^W_{t_i},\tau_{nt_i})_b,
\qquad&
L_4(\boldsymbol{v})\coloneqq&
S_2(u^W_{t_i},v_{t_i})_b,
\end{alignedat}
\end{cases}
\end{equation}
\endgroup
where coefficients $\{S_{i}\}_{i=1}^{8}$, $\{R_i\}_{i=1}^{4}$, $\{T_{i}\}_{i=1}^{2}$, and $C_1$ are related to $\{k_i\}$ and $\{m_{ij}\}$; see our supplementary material for detailed expressions.

Conversely, for Maxwell molecules, it shows $k_0=k_5=1$ and $k_1=k_2=k_3=k_4=0$, which means $\bbm{s}=\bm{s}$, $\bbm{\sigma}=\bm{\sigma}$. Since $\mathscr{A}_{1}$ isn't coercive for Maxwell molecules, we adopt the spaces $\mathrm{T}_2 \coloneqq \mathrm{V}_{\bm{\sigma}} \times \mathrm{V}_{\bm{s}} \times \mathrm{V}_p$ and $\mathrm{W} \coloneqq \mathrm{V}_{\bm{u}} \times \mathrm{V}_\theta$, regrouping the variables as:
\begin{equation}
\label{eq:grouping2}
\begin{cases}
\text{Unknowns:}\quad
\boldsymbol{\mathcal{S}} \coloneqq (\boldsymbol{\sigma}, \boldsymbol{s}, p) \in \mathrm{T}_2, 
\quad \boldsymbol{\mathcal{W}} \coloneqq (\boldsymbol{u}, \theta) \in \mathrm{W}, \\
\text{Test Functions:}\quad
\boldsymbol{\mathcal{R}}
\coloneqq (\boldsymbol{\tau},
\boldsymbol{r}, q) \in \mathrm{T}_2, \quad \boldsymbol{\mathcal{V}} \coloneqq (\boldsymbol{v}, \gamma) \in \mathrm{W}.
\end{cases}
\end{equation}
The weak formulation seeks $\bcal{S} \in \mathrm{T}_2$ and $\bcal{W} \in \mathrm{W}$ satisfying:
\begin{equation}
\label{eq:weak_for_M}
\begin{aligned}
  \mathscr{A}_2(\boldsymbol{\mathcal{S}},\boldsymbol{\mathcal{R}})+ \mathscr{B}_2(\boldsymbol{\mathcal{R}},\bcal{W}) &= \mathscr{F}_2(\boldsymbol{\mathcal{R}}) \quad &\forall\, \boldsymbol{\mathcal{R}} \in \mathrm{T}_2,
  \\
  \mathscr{B}_2(\boldsymbol{\mathcal{S}},\bcal{V}) &= 0 \quad &\forall\, \bcal{V} \in \mathrm{W},
\end{aligned}   
\end{equation}
with the block operators defined as:
\begin{equation}
    \label{eq:expression_for_operators_2}
\begin{cases}
\mathscr{A}_2(\boldsymbol{\mathcal{S}},\boldsymbol{\mathcal{R}})
\coloneqq
a(\boldsymbol{s},\boldsymbol{r})
+
c(\boldsymbol{s},\boldsymbol{\tau})
-
c(\boldsymbol{r},\boldsymbol{\sigma})
+
d(\boldsymbol{\sigma},\boldsymbol{\tau})
,
\\
\mathscr{B}_2(\boldsymbol{\mathcal{S}},\boldsymbol{\mathcal{V}})
\coloneqq
-
b(\gamma, \boldsymbol{s})
-
e(\boldsymbol{v},\boldsymbol{\sigma})
+
g(p,\boldsymbol{v})
,
\quad
\mathscr{F}_2(\boldsymbol{\mathcal{R}})
\coloneqq L_1(\boldsymbol{r})-L_3(\boldsymbol{\tau}).    
\end{cases}
\end{equation}
Finally, we introduce a unified notation for both regimes ($i \in \{1, 2\}$). Defining the global states $\bcal{U}_1 := (\bcal{S}_1, \bcal{W}_1)$ and $\bcal{U}_2 := (\bcal{S}_2, \bcal{W}_2)$ belonging to either $\mathrm{T}_1 \times V_p$ or $\mathrm{T}_2 \times \mathrm{W}$, the monolithic bilinear form $\mathcal{B}$ is formulated as:
\begin{equation}
    \label{eq:whole_bilinear_form}
    \mathcal{B}(\bcal{U}_1,\bcal{U}_2)\coloneqq\mathscr{A}_i(\bcal{S}_1,\bcal{S}_2)+\mathscr{B}_i(\bcal{S}_1,\bcal{W}_2)-\mathscr{B}_i(\bcal{S}_2,\bcal{W}_1).
\end{equation}

Before turning to the proof, we summarize the main conclusion of this section.
The precise function spaces and coefficient assumptions for the two molecular
regimes will be specified in the following subsections.

\begin{theorem}[Steady-state well-posedness]
\label{thm:steady-main}
Let $\Omega\subset\mathbb R^3$ be a bounded Lipschitz domain. Under the
assumptions stated in subsection \ref{sec:stablenM_wellposedness}
and \ref{subsec:stable_Maxwell_molecule}, 
the mixed formulations \eqref{eq:weak_for_nM} and \eqref{eq:weak_for_M} are well-posed. More precisely, they admit unique solutions $(\bcal S,p)\in \mathrm{T}_1\times \mathrm{V}_p$ and $(\bcal S,\bcal W)\in \mathrm{T}_2\times \mathrm{W}$ respectively, satisfying:
\[
    \|\bcal S\|_{\mathrm{T}_1}+\|p\|_{\mathrm{V}_p}
    \lesssim \|\mathscr{F}_1\|_{\mathrm{T}_1^*},
    \qquad
    \|\bcal S\|_{\mathrm{T}_2}+\|\bcal {W}\|_{\mathrm{W}}
    \lesssim \|\mathscr{F}_2\|_{\mathrm{T}_2^*}.
\]
\end{theorem}

The remainder of this section is devoted to the proof of
Theorem~\ref{thm:steady-main}. We first establish the boundary-related
Korn-type inequality needed for the coercivity estimates, and then verify the
Babu\v{s}ka--Brezzi conditions separately for the non-Maxwell and Maxwell cases.


\subsection{Boundary-related Korn-type inequalities} 
As the main analytic tool for the coercivity estimates in the mixed formulations, we introduce a specialized Korn-type inequality adapted to our specific boundary conditions. 
For spatial dimensions $d \ge 3$, Lewintan et al. \cite{lewintan2025wellposedness} established a Korn-type inequality for vector fields $\boldsymbol{u} \in H^1(\Omega; \mathbb{R}^d)$ as follows:
\begin{equation}
    \label{eq:Korn_type_inequality}
    \|\boldsymbol{u}\|_{0}^2+\|\operatorname{stf}\nabla\boldsymbol{u}\|_{0}^2\gtrsim \|\boldsymbol{u}\|_{1}^2.
\end{equation}
Our goal is to extend the estimate in \eqref{eq:Korn_type_inequality} by replacing the $L^2$ norm $\|\boldsymbol{u}\|_0$ with the boundary norm of $\boldsymbol{u}$, defined as $\|\boldsymbol{u}\|_b := \left(\int_{\partial\Omega} |\boldsymbol{u}|^2 \dif s\right)^{\frac{1}{2}}$. To this end, we first introduce the concept of the \textit{conformal Killing space} and establish a key property associated with it.
\begin{lemma}
    \label{Lem:CK_as_Level_Set}
    Let $\Omega \subset \mathbb{R}^d$ ($d \ge 3$) be a bounded Lipschitz domain. Define the \textit{conformal Killing space} as: 
    \begin{equation} 
    \label{eq:conformal_killing_map}
    \mathbf{CK}\coloneqq\{\boldsymbol{u}(\mathbf{x}):
        \boldsymbol{u}(\mathbf{x}) = \mathbf{a} + \lambda\mathbf{x} + A\mathbf{x} + (2(\mathbf{b}\cdot\mathbf{x})\mathbf{x} - \|\mathbf{x}\|^2\mathbf{b})\},
    \end{equation}
    where $\mathbf{a}, \mathbf{b} \in \mathbb{R}^d$, $\lambda \in \mathbb{R}$, and $A \in \mathbb{R}^{d\times d}$ is a skew-symmetric matrix. If $\boldsymbol{\varphi}\in \mathbf{CK}$ satisfies the homogeneous boundary condition $\boldsymbol{\varphi}|_{\partial\Omega}=\mathbf{0}$, then $\boldsymbol{\varphi}=\mathbf{0}$.
\end{lemma}
\begin{remark}
    \label{rmk:CK_space_prop}
    Defining the operator $\mathcal{E}:=\stf~\nabla$, it is established that $\mathbf{CK} = \ker \mathcal{E}$ for $d \ge 3$. A rigorous proof can be found in \cite[Proposition 2.5]{schirra2012new}.
\end{remark}
\begin{proof}
    By \cref{rmk:CK_space_prop}, $\boldsymbol{\varphi}\in\mathbf{CK}$ implies $\stf~\nabla\boldsymbol{\varphi}=\mathbf{0}$. Denoting the symmetric gradient by $\epsilon(\boldsymbol{\varphi})\coloneqq \frac{1}{2}(\nabla\boldsymbol{\varphi}+\nabla^{T}\boldsymbol{\varphi})$, we can express the trace-free part as $\stf~\nabla\boldsymbol{\varphi}=\epsilon(\boldsymbol{\varphi})-\frac{1}{d}(\nabla\cdot\boldsymbol{\varphi})I$. Consequently, we have:\vspace{-1ex}
    \begin{equation}
        \label{eq:stfvarphi_check}
        \begin{aligned}
        \|\stf~\nabla\boldsymbol{\varphi}\|_{F}^2&=\left(\epsilon(\boldsymbol{\varphi})-\frac{1}{d}(\nabla\cdot\boldsymbol{\varphi})I\right)\colon\left(\epsilon(\boldsymbol{\varphi})-\frac{1}{d}(\nabla\cdot\boldsymbol{\varphi})I\right)\\
        &=\|\epsilon(\boldsymbol{\varphi})\|_{F}^2-\frac{1}{d}(\nabla\cdot\boldsymbol{\varphi})^2=\left\|\frac{1}{2}(\nabla\boldsymbol{\varphi}+\nabla^{T}\boldsymbol{\varphi})\right\|_{F}^2-\frac{1}{d}(\nabla\cdot\boldsymbol{\varphi})^2\\
        &=\frac{1}{2}\|\nabla\boldsymbol{\varphi}\|_{F}^2+\frac{1}{2}(\nabla\boldsymbol{\varphi}\colon\nabla^{T}\boldsymbol{\varphi})-\frac{1}{d}(\nabla\cdot\boldsymbol{\varphi})^2=0.
        \end{aligned}
    \end{equation}
To evaluate the cross term  $\nabla\boldsymbol{\varphi}\colon\nabla^{T}\boldsymbol{\varphi}$, we apply integration by parts. Utilizing the boundary condition $\boldsymbol{\varphi}|_{\partial\Omega}=\mathbf{0}$, we obtain:\vspace{-1ex}
\begin{equation}
    \label{eq:cross_term}
    \begin{aligned}
        \int_{\Omega}\nabla\boldsymbol{\varphi}\colon\nabla^{T}\boldsymbol{\varphi}\dif\mathbf{x}&=\sum_{i,j=1}^{d}\int_{\Omega}\partial_{j}\varphi_i\partial_{i}\varphi_j\dif\mathbf{x}=\sum_{i,j=1}^{d}\int_{\Omega}(\partial_i(\varphi_j(\partial_j\varphi_i))-\varphi_j\partial_i\partial_j\varphi_i)\dif\mathbf{x}\\
        &=-\sum_{i,j=1}^{d}\int_{\Omega}\varphi_j\partial_i\partial_j\varphi_i\dif\mathbf{x}=-\sum_{i,j=1}^{d}\int_{\Omega}\varphi_j\partial_j\partial_i\varphi_i\dif\mathbf{x}\\
        &=\sum_{i,j=1}^{d}\int_{\Omega}(\partial_i\varphi_i)(\partial_j\varphi_j)\dif\mathbf{x}=\int_{\Omega}|\nabla\cdot\boldsymbol{\varphi}|^2\dif\mathbf{x}.
    \end{aligned}
\end{equation}
Integrating \eqref{eq:stfvarphi_check} over $\Omega$ and substituting \eqref{eq:cross_term} yields:
\begin{equation}
    \frac{1}{2}\int_{\Omega}\|\nabla\boldsymbol{\varphi}\|_{F}^2\dif\mathbf{x}+\left(\frac{1}{2}-\frac{1}{d}\right)\int_{\Omega}(\nabla\cdot\boldsymbol{\varphi})^2\dif\mathbf{x}=0.
\end{equation}
Since $d\ge 3$, we have $\frac{1}{2}-\frac{1}{d}>0$, this implies $\nabla\boldsymbol{\varphi}=\mathbf{0}$. Since $\boldsymbol{\varphi}|_{\partial\Omega}=\boldsymbol{0}$, we conclude that $\boldsymbol{\varphi}=\mathbf{0}$ in $\Omega$.
\end{proof}

With these preliminaries, we are now positioned to establish the \textit{boundary-related Korn-type inequality}.
\begin{theorem}
    \label{Thm:Boundary_related_Korn}
     Let $\Omega \subset \mathbb{R}^d$ be a bounded Lipschitz domain with $d \ge 3$. For $\boldsymbol{u}\in H^1(\Omega;\mathbb{R}^d)$, the following Korn-type Inequality holds:
    \begin{equation}
    \label{eq:Korn_type_rel_to_bdry}
    \|\boldsymbol{u}\|_{b}^2+\|\operatorname{stf}\nabla \boldsymbol{u}\|_{0}^2\gtrsim \|\boldsymbol{u}\|_1^2
    \end{equation}
\end{theorem}
\begin{proof}
    We proceed by contradiction. Suppose  \eqref{eq:Korn_type_rel_to_bdry} does not hold, then there exists a sequence $\{\boldsymbol{u}_n\}\subset H^{1}(\Omega;\mathbb{R}^{d})$ such that $\|\boldsymbol{u}_n\|_{1}=1$ and 
    \begin{equation}
    \label{eq:Counter}
    \lim_{n\rightarrow\infty} (\|\boldsymbol{u}_n\|_{b}^2+\|\operatorname{stf}\nabla\boldsymbol{u}_n\|_{0}^2)=0.
    \end{equation}
    By the Rellich–Kondrachov theorem \cite[Section 5.7]{evans2010partial}, the embedding $H^{1}(\Omega;\mathbb{R}^{d})\subset\subset L^2(\Omega;\mathbb{R}^d)$, is compact. Thus, there exists a subsequence $\{\boldsymbol{u}_{n_k}\}$ that converges strongly to some limit $\boldsymbol{u}_0$ in $L^2(\Omega; \mathbb{R}^d)$. Consequently, $\{\boldsymbol{u}_{n_k}\}$ is a Cauchy sequence in $L^2(\Omega; \mathbb{R}^d)$.
    
    Furthermore, utilizing the classical Korn-type inequality \eqref{eq:Korn_type_inequality}, for $k\neq l$ we have:
    $$
    \|\boldsymbol{u}_{n_k}-\boldsymbol{u}_{n_l}\|_{1}\lesssim \|\boldsymbol{u}_{n_k}-\boldsymbol{u}_{n_l}\|_0+\|\operatorname{stf}\nabla(\boldsymbol{u}_{n_k}-\boldsymbol{u}_{n_l})\|_{0}.
    $$
    Since $\lim_{n\rightarrow\infty}\|\text{stf}~\nabla\boldsymbol{u}_n\|_{0}=0$, it follows that 
    $\lim_{k,l\rightarrow\infty}\|\text{stf}~\nabla(\boldsymbol{u}_{n_k}-\boldsymbol{u}_{n_l})\|_{0}=0$. Combined with the fact that $\{\boldsymbol{u}_{n_k}\}$ is Cauchy in $L^2(\Omega; \mathbb{R}^d)$ (which implies $\lim_{k,l \to \infty} \|\boldsymbol{u}_{n_k} - \boldsymbol{u}_{n_l}\|_0 = 0$), we deduce that $\{\boldsymbol{u}_{n_k}\}$ is a Cauchy sequence in $H^1(\Omega; \mathbb{R}^d)$. Therefore, $\boldsymbol{u}_{n_k}$ converges strongly to $\boldsymbol{u}_0$ in $H^1(\Omega; \mathbb{R}^d)$.

    Passing to the limit in \eqref{eq:Counter} yields
    \begin{equation*}
        \|\boldsymbol{u}_0\|_{b}^2+\|\text{stf}~\nabla\boldsymbol{u}_{0}\|_{0}^2=0\quad\Rightarrow \quad \operatorname{stf}\nabla\boldsymbol{u}_0=0,\quad \boldsymbol{u}_0=0\text{ on }\partial\Omega.
    \end{equation*}
    According to \cref{rmk:CK_space_prop}, $\stf~\nabla\boldsymbol{u}_0=0$ implies $\boldsymbol{u}_{0}\in\mathbf{CK}$. Since $\boldsymbol{u}_{0}=0$ on $\partial\Omega$, \cref{Lem:CK_as_Level_Set} guarantees that $\boldsymbol{u}_{0}=\mathbf{0}$ in $\Omega$. The strong convergence in $H^1$ ensures that $\|\boldsymbol{u}_0\|_1 = \lim_{k \to \infty} \|\boldsymbol{u}_{n_k}\|_1 = 1$, which is a contradiction. Thus, the Korn-type inequality \eqref{eq:Korn_type_rel_to_bdry} holds for $d\ge 3$. 
\end{proof}

\subsection{Non-Maxwell gas molecules}
\label{sec:stablenM_wellposedness}
We first treat the non-Maxwell case, where the full regularized R13 system contains second-order dissipative terms for all variables involved in the coercivity argument. Accordingly, the natural setting is an \(H^1\)-based mixed formulation for \(\boldsymbol{s},\boldsymbol{u},\boldsymbol{\sigma}\), and \(\theta\). We establish well-posedness for our R13 system under some constraints on the parameters $k_{i}$. First, we introduce the corresponding function spaces for the weak form \eqref{eq:weak_for_nM}:
\begin{equation}
     \label{eq:continuous_spaces}
\begin{aligned}
    &\mathrm{V}_{\boldsymbol{\sigma}} \coloneqq H^1(\Omega;\mathbb{R}_{\stf}^{3\times 3}),\quad \mathrm{V}_{\boldsymbol{s}} \coloneqq H^1(\Omega;\mathbb{R}^3),\\ 
    &\mathrm{V}_{\boldsymbol{u}} \coloneqq \tilde{H}^1(\Omega;\mathbb{R}^3),\quad \mathrm{V}_{\theta} \coloneqq H^1(\Omega;\mathbb{R}),\quad \mathrm{V}_{p}\coloneqq L_{0}^{2}(\Omega;\mathbb{R}),
\end{aligned}
\end{equation}
where $\tilde{H}^{1}(\Omega;\mathbb{R}^{3})\coloneqq\{\boldsymbol{u}\in H^1(\Omega;\mathbb{R}^{3}):u_{n}=0\text{ on }\partial\Omega\}$. Since the operator $\mathscr{A}_{1}$ and $\mathscr{B}_{1}$ are both continuous (see \cite[Appendix A]{lewintan2025wellposedness}), we just need to show the coercivity of $\mathscr{A}_{1}$ and the inf-sup condition for $\mathscr{B}_{1}$.
\begin{theorem}[Inf-Sup Condition]
    \label{Thm:inf_sup_nonMax}
    For a bounded Lipschitz domain $\Omega \subset \mathbb{R}^3$, the bilinear form $\mathscr{B}_1$ satisfies the following inf-sup condition:
    \begin{equation}
        \label{eq:infsup_condition4g}
        \sup_{\boldsymbol{\mathcal{S}}\in \mathrm{T}_1}\frac{\left|\mathscr{B}_{1}(\boldsymbol{\mathcal{S}},q)\right|}{\|\boldsymbol{\mathcal{S}}\|_{\mathrm{T}_1}}\gtrsim\|q\|_{0},\quad\forall q\in L_{0}^2(\Omega).
    \end{equation}
\end{theorem}
\begin{proof}
    Recall that $\mathscr{B}_1(\boldsymbol{\mathcal{S}}, q) = \int_\Omega q\nabla\cdot\boldsymbol{u} \mathrm{d}\mathbf{x}$, from \cite[Lemma III.3.1]{galdi2011introduction} there exists $C>0$ such that: 
    $$\forall q\in L_{0}^{2}(\Omega),\quad\exists \boldsymbol{u}_{q}\in H_{0}^{1}(\Omega;\mathbb{R}^3)\subset\mathrm{V}_{\boldsymbol{u}}
    \text{ s.t. }\quad\frac{|\int_{\Omega}q\nabla\cdot\boldsymbol{u}_{q}\dif\mathbf{x}|}{\|\boldsymbol{u}_{q}\|_{1}}\ge C\|q\|_{0}.
    $$
    For $q \in L_0^2(\Omega)$, constructing the composite test function $\boldsymbol{\mathcal{S}}\coloneqq(\mathbf{0}, \boldsymbol{u}_q, \mathbf{0}, 0) \in \mathrm{T}_1$ and substituting it into the supremum directly yields \eqref{eq:infsup_condition4g}.
\end{proof}

We now examine the coercivity of the bilinear form $\mathscr{A}_1$. Recalling from \eqref{eq:expression_for_operators_1}, we can decompose $\mathscr{A}_1(\boldsymbol{\mathcal{S}}, \boldsymbol{\mathcal{S}})$ as: 
\begin{equation}
    \label{eq:expression_for_A}
    \begin{aligned}
    &\mathscr{A}_{1}(\boldsymbol{\mathcal{S}},\boldsymbol{\mathcal{S}})=\underbrace{(a(\bar{\boldsymbol{s}},\bar{\boldsymbol{s}})+2j(\bar{\boldsymbol{s}},\boldsymbol{u})+f(\boldsymbol{u},\boldsymbol{u}))}_{J_1}+\underbrace{(d(\bar{\boldsymbol{\sigma}},\bar{\boldsymbol{\sigma}})+2z(\theta,\bar{\boldsymbol{\sigma}})+h(\theta,\theta))}_{J_2}.
    \end{aligned}
\end{equation}
By the definitions of these bilinear forms \eqref{eq:detailed_bilinear_forms}, $\mathscr{A}_1(\boldsymbol{\mathcal{S}},\boldsymbol{\mathcal{S}})$ contains both integrals on $\Omega$ and $\partial \Omega$.
Below, we collect all boundary integrals in $\mathscr{A}_1(\boldsymbol{\mathcal{S}},\boldsymbol{\mathcal{S}})$ and define the sum of these terms as $\mathscr{A}_{\mathrm{bdry}}(\boldsymbol{\mathcal{S}},\boldsymbol{\mathcal{S}})$, whose precise definition is
\begin{equation}
    \label{eq:Coercivity_on_bdry}
    \begin{aligned}
    \mathscr{A}_{\mathrm{bdry}}(\boldsymbol{\mathcal{S}},\boldsymbol{\mathcal{S}})&=(S_1\textstyle\sum_{i=1}^2\|\bar{s}_{t_i}\|_{b}^2+S_2\textstyle\sum_{i=1}^2\|u_{t_i}\|_{b}^2+2T_1\textstyle\sum_{i=1}^2(\bar{s}_{t_i},u_{t_i})_{b})\\
    &+(S_3\|\bar{\sigma}_{nn}\|_{b}^2+S_4\|\theta\|_{b}^2+2T_2(\bar{\sigma}_{nn},\theta)_{b})\\
    &+S_5\|s_n\|_{b}^2+S_{6}\left\|\bar{\sigma}_{t_1t_1}+\frac{1}{2}\bar{\sigma}_{nn}\right\|_{b}^2+S_7\textstyle\sum_{i=1}^2\|\bar{\sigma}_{nt_i}\|_{b}^2+S_8\|\bar{\sigma}_{t_1t_2}\|_{b}^2.
\end{aligned}
\end{equation}
We will first show that $\mathscr{A}_{\mathrm{bdry}}(\boldsymbol{\mathcal{S}},\boldsymbol{\mathcal{S}})$ is nonnegative:

\begin{lemma}
    \label{Lem:coefficients_on_bdry}
    The coefficients in the definition of $\mathscr{A}_{\mathrm{bdry}}(\boldsymbol{\mathcal{S}},\boldsymbol{\mathcal{S}})$ satisfy:
    \begin{equation}
        \label{eq:Coefficients_for_boundary_terms}
        S_{i}\ge 0\;(1\le i\le 8),\quad T_{1}^2\le S_1S_2,\quad T_{2}^2\le S_3S_4.
    \end{equation}    
\end{lemma}
\begin{proof}
    Comparing the diagonal evaluation of the steady weak formulation with the integration-by-parts identity derived in \cref{Sec:Entropy}, we obtain 
    \begin{equation}
        \label{eq:equivalence}
        \mathcal{B}(\bcal{U},\bcal{U})-\mathscr{F}_1(\bcal{S})=\mathscr{A}_{1}(\boldsymbol{\mathcal{S}},\boldsymbol{\mathcal{S}})-\mathscr{F}_{1}(\boldsymbol{\mathcal{S}})=\int_{\Omega}\innerprod{\bcal{U}}{(\mathcal{A}(\nabla)-\mathcal{D}(\nabla,\nabla)-S)\bcal{U}}_{\mathbf{V}}\dif\mathbf{x}.
    \end{equation}
    Based on the kinetic derivations in  \cref{subsec:Bdry_Term_for_Entropy}, this equivalence implies:
    \begin{equation}
        \mathscr{A}_{\mathrm{bdry}}(\boldsymbol{\mathcal{S}},\boldsymbol{\mathcal{S}})-\mathscr{F}_{1}(\boldsymbol{\mathcal{S}})=I_{\mathrm{bdry}}=
        2(\tilde{Q}\mathcal{A}_{n}f_{W},\mathcal{A}_{n}f_{\even})_b-2(\tilde{Q}\mathcal{A}_{n}f_{\even},\mathcal{A}_{n}f_{\even
        })_b.
    \end{equation}
    Note that $\mathscr{F}_1(\mathcal{S})$ is a linear functional dependent on the boundary data $\boldsymbol{u}^W$ and $\theta^W$. Conversely, the quadratic form $\mathscr{A}_{\mathrm{bdry}}(\bcal{S}, \bcal{S})$ is entirely independent of these terms. By separating the terms accordingly, we obtain $\mathscr{F}_1(\mathcal{S}) = -2(\tilde{Q}\mathcal{A}_n f_W, \mathcal{A}_n f_{\even})$ and
$$
\forall\boldsymbol{\mathcal{S}}\in \mathrm{T}_1,\quad
\mathscr{A}_{\mathrm{bdry}}(\boldsymbol{\mathcal{S}},\boldsymbol{\mathcal{S}})=-2(\tilde{Q}\mathcal{A}_{n}f_{\even},\mathcal{A}_{n}f_{\even})_b\ge0.
$$
It implies \eqref{eq:Coercivity_on_bdry} is positive semi-definite, which yields \eqref{eq:Coefficients_for_boundary_terms}.
\end{proof}

Next, we establish the coercivity estimates for $J_1$ and $J_2$ in \eqref{eq:expression_for_A}.
\begin{theorem}
    \label{Thm:Coercive_Part_1}
    Under the strict constraint $25k_4^2 < 24k_3k_7$, the sub-form $J_1$ is coercive, i.e.,
    \begin{equation}
        \label{eq:Coercivity_1}
        J_{1}=a(\bar{\boldsymbol{s}},\bar{\boldsymbol{s}})+2j(\bar{\boldsymbol{s}},\boldsymbol{u})+f(\boldsymbol{u},\boldsymbol{u})\gtrsim \|\bar{\boldsymbol{s}}\|_{1}^2+\|\boldsymbol{u}\|_{1}^2.
    \end{equation}
\end{theorem}
\begin{proof}
    First, recall the Korn-type inequalities established previously:
    \begin{equation}
        \label{eq:Korn-type-apply}
        \|\operatorname{stf}\nabla\boldsymbol{s}\|_{0}^2+\|\boldsymbol{s}\|_{0}^2\gtrsim \|\boldsymbol{s}\|_{1}^2,\quad \|\operatorname{stf}\nabla\boldsymbol{u}\|_{0}^2+\|\boldsymbol{u}\|_{b}^2\gtrsim\|\boldsymbol{u}\|_{1}^2,
    \end{equation}
    Due to the non-penetration condition $u_n = 0$, the boundary norm reduces precisely to its tangential components: $\sum_{i=1}^2 \|u_{t_i}\|_b^2 = \|\boldsymbol{u}\|_{L^2(\partial\Omega)}^2$.
    
    Utilizing the parameter constraint $25k_4^2 < 24k_3k_7$ and applying Young's inequality to the cross terms, the volume integrals in $J_1$ can be bounded from below by $\|\stf~\nabla\bar{\boldsymbol{s}}\|_0^2 + \|\stf~\nabla\boldsymbol{u}\|_0^2 + \|\bar{\boldsymbol{s}}\|_0^2$, which sequentially bounds $\|\bar{\boldsymbol{s}}\|_1^2 + \|\stf~\nabla\boldsymbol{u}\|_0^2$. 

    Furthermore, the trace inequality ensures $\|\bar{\boldsymbol{s}}\|_1^2 \gtrsim \|\bar{\boldsymbol{s}}\|_b^2$. This allows us to extract a small parameter $\delta > 0$ from the $H^1$ volume control to augment the boundary coefficient $S_1$, thereby guaranteeing the strict positive definiteness of the boundary quadratic form. Consequently, we have:
    \begin{equation}
        \label{eq:J1_estimation}
        \begin{aligned}
        J_1&\gtrsim\|\bar{\boldsymbol{s}}\|_{1}^2+\|\stf~\nabla\boldsymbol{u}\|_{0}^2+\sum_{i=1}^{2}((S_1+\delta)\|\bar{s}_{t_i}\|_{b}^2+2T_1(\bar{s}_{t_i},u_{t_i})_{0}+S_2\|u_{t_i}\|_{b}^2)\\        &\gtrsim \|\bar{\boldsymbol{s}}\|_{1}^2+\|\stf~\nabla\boldsymbol{u}\|_{0}^2+\|\boldsymbol{u}\|_{b}^2\gtrsim\|\bar{\boldsymbol{s}}\|_{1}^2+\|\boldsymbol{u}\|_{1}^2.
        \end{aligned}
    \end{equation}
\end{proof}
\begin{theorem}
    \label{Thm:Coercive_Part_2}
    Under the strict constraint $3k_2^2<k_1k_{10}$, the sub-form $J_2$ is coercive, i.e., 
    \begin{equation}
        \label{eq:Coercivity_2}
        d(\bar{\boldsymbol{\sigma}},\bar{\boldsymbol{\sigma}})+2z(\theta,\bar{\boldsymbol{\sigma}})+h(\theta,\theta)\gtrsim \|\bar{\boldsymbol{\sigma}}\|_{1}^2+\|\theta\|_{1}^2.
    \end{equation}
\end{theorem}
\begin{proof}
    The proof proceeds analogously to that of \cref{Thm:Coercive_Part_1}. Applying Young's inequality under the condition $3k_2^2 < k_1 k_{10}$ bounds the volume integrals of $J_2$ from below by $\|\stf \nabla\bar{\boldsymbol{\sigma}}\|_0^2 + \|\bar{\boldsymbol{\sigma}}\|_0^2 + \|\nabla\theta\|_0^2$. By the Korn-type inequality in \cite{lewintan2025wellposedness}, this sum subsequently controls $\|\bar{\boldsymbol{\sigma}}\|_1^2 + \|\nabla\theta\|_0^2$. As before, the trace inequality $\|\bar{\boldsymbol{\sigma}}\|_1^2 \gtrsim \|\bar{\boldsymbol{\sigma}}\|_b^2$ enables us to extract a sufficiently small parameter $\delta > 0$ to guarantee the strict positive definiteness of the boundary quadratic form governed by the coefficients $S_3$, $S_4$, $T_2$ and $\tilde{T}_2$. Integrating the generalized Poincaré-Friedrichs inequality (which asserts $\|\nabla\theta\|_0^2 + \|\theta\|_b^2 \gtrsim \|\theta\|_1^2$), we conclude:
    \begin{equation}
        \label{eq:J2_estimation}
        \begin{aligned}
        J_2&\gtrsim\|\bar{\boldsymbol{\sigma}}\|_{1}^2+\|\nabla\theta\|_{0}^2+(S_3+\delta)\|\sigma_{nn}\|_{b}^2+2T_{2}(\theta,\sigma_{nn})_{0}+S_4\|\theta\|_{b}^2\\
        &\gtrsim \|\bar{\boldsymbol{\sigma}}\|_{1}^2+\|\nabla\theta\|_{0}^2+\|\theta\|_{b}^2\gtrsim\|\bar{\boldsymbol{\sigma}}\|_{1}^2+\|\theta\|_{1}^2.
        \end{aligned}
    \end{equation}
\end{proof}

Together, \cref{Thm:Coercive_Part_1} and \cref{Thm:Coercive_Part_2} establish the coercivity of $\mathscr{A}_1$, while \cref{Thm:inf_sup_nonMax} ensures the inf-sup condition for $\mathscr{B}_1$. Hence, under the strict constraints stated earlier, \cref{thm:steady-main} holds for non-Maxwell molecules. 
\begin{remark}
    The constraints $25k_4^2<24k_3k_7$ and $3k_2^2<k_1k_{10}$ for non-Maxwell molecules are numerically validated in \cref{appA:Coefficients_for_k}. They fail in the Maxwell case, so the above argument does not extend directly. This only reflects the limitation of the present coercivity framework, rather than any ill-posedness of the Maxwell model. A separate analysis will be given in the next subsection.
\end{remark}
\subsection{Maxwell gas molecules}
\label{subsec:stable_Maxwell_molecule}
For Maxwell molecules, the loss of coercivity in \(\mathscr A_1\) necessitates the
alternative weak formulation \eqref{eq:weak_for_M} and redefined spaces:
\begin{equation}
     \label{eq:continuous_spaces_Maxwell}
\begin{aligned}
    &\mathrm{V}_{\boldsymbol{\sigma}} \coloneqq H^1(\Omega;\mathbb{R}_{\stf}^{3\times 3}),\quad \mathrm{V}_{\boldsymbol{s}} \coloneqq H^1(\Omega;\mathbb{R}^3),\\ 
    &\mathrm{V}_{\boldsymbol{u}} \coloneqq H_{0}(\mathrm{div},\Omega),\quad \mathrm{V}_{\theta} \coloneqq L^2(\Omega;\mathbb{R}),\quad \mathrm{V}_{p}\coloneqq L_{0}^{2}(\Omega;\mathbb{R}),
\end{aligned}
\end{equation}
where $H_{0}(\mathrm{div},\Omega)\coloneqq \{\boldsymbol{u}\in H(\mathrm{div},\Omega):\boldsymbol{u}\cdot\boldsymbol{n}=0\text{ on }\partial\Omega\}$. 
Lewintan et al.~\cite{lewintan2025wellposedness} established the well-posedness of the Maxwell case in a continuous setting with
\(\mathrm{V}_{\bm{u}}=L^2(\Omega;\mathbb R^3)\) and
\(\mathrm{V}_p=\{p\in H^1(\Omega):\int_\Omega p\,dx=0\}\).
While this choice is natural for boundary conditions allowing penetration, it
would impose unnecessarily strong regularity on the pressure under the
non-penetration condition considered here, possibly leading to mismatched approximation orders in finite element discretizations; see
\cite{hu2026stabilityconvergencemixedfinite}.
We instead use
\(\mathrm{V}_p=L_0^2(\Omega)\) and \(\mathrm{V}_{\bm{u}}=H_0(\mathrm{div},\Omega)\). Under this alternative functional setting, we prove the well-posedness of the Maxwell case.
\begin{theorem}
\label{thm:inf_sup_modified_Maxwell}
    For a bounded Lipschitz domain $\Omega \subset \mathbb{R}^3$, the operator $\mathscr{B}_2$ satisfies inf-sup condition, i.e. 
    \begin{equation}
        \label{eq:new_inf_sup_Maxwell}
        \sup_{\boldsymbol{\mathcal{R}}\in\mathrm{T}_2}\frac{|\mathscr{B}_{2}(\boldsymbol{\mathcal{R}},\boldsymbol{\mathcal{V}})|}{\|\boldsymbol{\mathcal{R}}\|_{\mathrm{T}_2}}\gtrsim \|\boldsymbol{\mathcal{V}}\|_{\mathrm{W}}.
    \end{equation}
\end{theorem}
\begin{proof}
    Recall the explicit formulation of $\mathscr{B}_{2}$ as:
    \begin{equation}    \mathscr{B}_2(\boldsymbol{\mathcal{R}},\boldsymbol{\mathcal{V}})=\int_{\Omega}(q\nabla\cdot\boldsymbol{v}-\boldsymbol{v}\cdot(\nabla\cdot\boldsymbol{\tau})-\gamma\nabla\cdot\boldsymbol{r})\dif\mathbf{x}.
    \end{equation}
    For any $\boldsymbol{\mathcal{V}}=(\boldsymbol{v}, \gamma) \in \mathrm{W}$, we construct a specific test function $\tilde{\boldsymbol{\mathcal{R}}}\coloneqq(\boldsymbol{\tau}, \boldsymbol{r}, q) \in \mathrm{T}_2$. Setting $q := \nabla \cdot \boldsymbol{v}$ ensures $q \in L_0^2(\Omega)$ due to the boundary condition $\boldsymbol{v} \cdot \boldsymbol{n} = 0$. 
    Furthermore, the surjectivity of the divergence operator \cite{lewintan2025wellposedness} guarantees the existence of $\boldsymbol{\tau}$ and $\boldsymbol{r}$ such that $-\nabla \cdot \boldsymbol{\tau} = \boldsymbol{v}$ and $-\nabla \cdot \boldsymbol{r} = \gamma$, satisfying the stability bounds $\|\boldsymbol{\tau}\|_1 \lesssim \|\boldsymbol{v}\|_0$ and $\|\boldsymbol{r}\|_1 \lesssim \|\gamma\|_0$.

    This construction directly yields $\mathscr{B}_2(\tilde{\boldsymbol{\mathcal{R}}}, \boldsymbol{\mathcal{V}}) = \|\boldsymbol{\mathcal{V}}\|_{\mathrm{W}}^2$.Since the stability bounds concurrently ensure $\|\tilde{\boldsymbol{\mathcal{R}}}\|_{\mathrm{T}_2} \lesssim \|\boldsymbol{\mathcal{V}}\|_{\mathrm{W}}$, the inf-sup condition follows immediately:
    \begin{equation}
    \sup_{\boldsymbol{\mathcal{R}} \in \mathrm{T}_2} \frac{|\mathscr{B}_2(\boldsymbol{\mathcal{R}}, \boldsymbol{\mathcal{V}})|}{\|\boldsymbol{\mathcal{R}}\|_{\mathrm{T}_2}} \ge \frac{\mathscr{B}_2(\tilde{\boldsymbol{\mathcal{R}}}, \boldsymbol{\mathcal{V}})}{\|\tilde{\boldsymbol{\mathcal{R}}}\|_{\mathrm{T}_2}} \gtrsim \frac{\|\boldsymbol{\mathcal{V}}\|_{\mathrm{W}}^2}{\|\boldsymbol{\mathcal{V}}\|_{\mathrm{W}}} = \|\boldsymbol{\mathcal{V}}\|_{\mathrm{W}}. 
    \end{equation}
\end{proof}
To complete the well-posedness analysis, we now establish the coercivity of $\mathscr{A}_2$ on $\ker\mathscr{B}_2$. We begin by characterizing this kernel space.
\begin{lemma}
\label{Lem:KernelB}
    If $\boldsymbol{\mathcal{S}}= (\boldsymbol{\sigma}, \boldsymbol{s}, p) \in \ker\mathscr{B}_2$, then $p \in H^1(\Omega) \cap L_0^2(\Omega)$ with $\nabla p = -\nabla \cdot \boldsymbol{\sigma}$, and $\nabla \cdot \boldsymbol{s} = 0$.
\end{lemma}
\begin{proof}
    By definition, $\boldsymbol{\mathcal{S}}\in \ker\mathscr{B}_2$ implies $\int_\Omega (p \nabla \cdot \boldsymbol{v} - \boldsymbol{v} \cdot (\nabla \cdot \boldsymbol{\sigma}) - \gamma \nabla \cdot \boldsymbol{s}) d\mathbf{x} = 0$ for all $(\boldsymbol{v}, \gamma) \in \mathrm{W}$. Testing this identity with $(\boldsymbol{0}, \gamma)$ immediately yields $\nabla \cdot \boldsymbol{s} = 0$. Subsequently, testing against $(\boldsymbol{v}, 0)$ for any $\boldsymbol{v} \in C_0^\infty(\Omega; \mathbb{R}^3) \subset H_0(\text{div}, \Omega)$ gives $\int_\Omega p \nabla \cdot \boldsymbol{v} \dif\mathbf{x} = \int_\Omega \boldsymbol{v} \cdot (\nabla \cdot \boldsymbol{\sigma}) \dif\mathbf{x}$. This perfectly matches the weak derivative definition, implying $\nabla p = -\nabla \cdot \boldsymbol{\sigma} \in L^2(\Omega; \mathbb{R}^3)$, which confirms $p \in H^1(\Omega)$.
\end{proof}
Given $p \in L_0^2(\Omega)$, \cref{Lem:KernelB} and Poincaré's inequality allow us to bound the pressure norm via the stress tensor strictly:
\begin{equation}
    \label{eq:estimate_for_p_L2}
    \|p\|_{0}\lesssim \|\nabla p\|_{0}=\|\nabla\cdot\boldsymbol{\sigma}\|_{0}\lesssim\|\boldsymbol{\sigma}\|_{1}.
\end{equation}
Coupling \eqref{eq:estimate_for_p_L2} with the structural coercivity established in \cite[Theorem 4.3]{lewintan2025wellposedness}, we directly arrive at the following coercivity result for the kernel space.
\begin{theorem}
    \label{Thm:Coercivity_NOW}
    The bilinear form $\mathscr{A}_2$ is coercive on the space $\ker\mathscr{B}_2$, i.e. $\forall\boldsymbol{\mathcal{S}}\in\ker\mathscr{B}_2$, $\mathscr{A}_2(\boldsymbol{\mathcal{S}},\boldsymbol{\mathcal{S}})\gtrsim\|\boldsymbol{\mathcal{S}}\|_{\mathrm{T}}^{2}$.
\end{theorem}

Finally, \cref{thm:inf_sup_modified_Maxwell} and \cref{Thm:Coercivity_NOW} collectively guarantee \cref{thm:steady-main} for Maxwell molecules.

\section{Well-posedness of time dependent problems}
\label{Sec:TimeDependent}
In this section, we establish the global well-posedness of the time-dependent problem \eqref{eq:time_dependent_compact_R13}. Specifically, for the homogeneous case ($\mathbf{g}_{\mathrm{ext}} = \mathbf{0}$), the existence and uniqueness of solutions are proved by invoking the Lumer-Phillips theorem \cite[Section 1.4]{pazy2012semigroups}. 
\subsection{Abstract operator formulation}
To apply the Lumer-Phillips theorem, we must first formulate the underlying abstract Cauchy problem. Let us introduce the base Hilbert space:
\begin{equation}
\mathcal{H}_{0}\coloneqq L^2(\Omega;\mathbb{R})\times L^2(\Omega;\mathbb{R})\times L^2(\Omega;\mathbb{R}^3)\times L^2(\Omega;\mathbb{R}^3)\times L^2(\Omega;\mathbb{R}_{\stf}^{3\times 3}),
\end{equation}
For any $t \ge 0$, the state vector is given by $\boldsymbol{\mathcal{U}}(\cdot, t) := (\rho, \theta, \boldsymbol{u}, \bar{\boldsymbol{s}}, \bar{\boldsymbol{\sigma}})^{T} \in \mathcal{H}_0$, and the corresponding inner product for any $\boldsymbol{\mathcal{U}}_1, \boldsymbol{\mathcal{U}}_2 \in \mathcal{H}_0$ is defined as:
\begin{equation}
    \label{eq:inner_product_H0}
    \innerprod{\boldsymbol{\mathcal{U}}_{1}}{\boldsymbol{\mathcal{U}_2}}_{\mathcal{H}_0}\coloneqq\int_{\Omega}\innerprod{\boldsymbol{\mathcal{U}}_1}{M\boldsymbol{\mathcal{U}}_2}_{\mathbf{V}}\dif\mathbf{x},
\end{equation}
Next, we specify the function spaces for the state variables depending on the molecular model. For non-Maxwell molecules, we define:
\begin{equation}
    \label{eq:space_1}
    \widetilde{\mathrm{U}}_{1}\coloneqq \mathrm{T}_{1}\times \mathrm{V}_{\rho}\quad\text{with}\quad \mathrm{V}_{\rho}\coloneqq L^2(\Omega;\mathbb{R}),
\end{equation}
where $\mathrm{T}_1$ is defined in \eqref{eq:grouping1} and \eqref{eq:continuous_spaces}. Conversely, for Maxwell molecules, the functional framework is adjusted to:
\begin{equation}
    \label{eq:space_2}
    \widetilde{\mathrm{U}}_2\coloneqq \mathrm{T}_2\times\mathrm{W}_2,\quad\text{where}\quad\mathrm{W}_2\coloneqq \mathrm{V}_{\boldsymbol{u}}\times\mathrm{V}_{\theta}\times \mathrm{V}_{\rho},\quad\text{and}\quad\mathrm{V}_{\rho}\coloneqq L^2(\Omega;\mathbb{R}).
\end{equation}
The constituent spaces $\mathrm{T}_2$, $\mathrm{V}_{\theta}$ and $\mathrm{V}_{\boldsymbol{u}}$ are established in \eqref{eq:grouping2} and \eqref{eq:continuous_spaces_Maxwell}. 
\begin{remark}
    Comparing \eqref{eq:space_1}, \eqref{eq:space_2} with \eqref{eq:grouping1}, \eqref{eq:grouping2}, the sole distinction is the omission of the zero-mean constraint for the pressure $p := \rho + \theta$. In the time-dependent regime, the mass conservation law guarantees that $\int_\Omega \rho(\mathbf{x}, t) \dif\mathbf{x} \equiv \int_\Omega \rho_0(\mathbf{x}) \dif\mathbf{x}$ for all $t \in I$. Since the mean value of $\rho(\cdot, t)$ is inherently determined by the initial state $\rho_0$, the zero-mean restriction is no longer required.
\end{remark}

For $i \in \{1, 2\}$, we construct the Gelfand triple $\widetilde{\mathrm{U}}_{i}\subset\mathcal{H}_{0}\cong\mathcal{H}_{0}^{*}\subset\widetilde{\mathrm{U}}_{i}^{*}$, where the state space $\mathcal{H}_0$ acts as the pivot space. 
Consequently, the duality pairing $\langle \cdot, \cdot \rangle_{\widetilde{\mathrm{U}}_i^*, \widetilde{\mathrm{U}}_i}$ serves as the unique continuous extension of the $\mathcal{H}_0$-inner product, meaning the identity $\langle \bm{\mathcal{U}}_1, \bm{\mathcal{U}}_2 \rangle_{\widetilde{\mathrm{U}}_i^*, \widetilde{\mathrm{U}}_i} = \langle \bm{\mathcal{U}}_1, \bm{\mathcal{U}}_2 \rangle_{\mathcal{H}_0}$ holds naturally for any $\bm{\mathcal{U}}_1, \bm{\mathcal{U}}_2 \in \widetilde{\mathrm{U}}_i$.

Recalling \eqref{eq:time_dependent_compact_R13}, we define the formal differential operator $\mathcal{L} := M^{-1}(\mathcal{D}(\nabla, \nabla) + S - \mathcal{A}(\nabla))$. Under the homogeneous boundary condition $\mathbf{g}_{\text{ext}} = \mathbf{0}$ and assuming a sufficiently smooth state $\bm{\mathcal{U}}$, applying the divergence theorem to the evolution equation $\bm{\mathcal{U}}_{t}=\mathcal{L}\bm{\mathcal{U}}$ yields: 
\begin{equation}
    \label{eq:div}
    \innerprod{\bm{\mathcal{U}}_{t}}{\bm{\mathcal{V}}}_{\mathcal{H}_0}=\innerprod{\mathcal{L}\bm{\mathcal{U}}}{\mathcal{V}}_{\mathcal{H}_0}=-\mathcal{B}(\bm{\mathcal{U}},\bm{\mathcal{V}}),\quad\forall i=1,2,
\end{equation}
where the bilinear form $\mathcal{B}$ is defined in \eqref{eq:whole_bilinear_form}. This motivates the definition of the corresponding weak operator $\mathcal{L}_{W,i}:\widetilde{\mathrm{U}}_i\rightarrow\widetilde{\mathrm{U}}_{i}^{*}$ as:
\begin{equation}
    \label{eq:weak_L}
    \langle \mathcal{L}_{W,i}\boldsymbol{\mathcal{U}}, \boldsymbol{\mathcal{V}} \rangle_{\widetilde{\mathrm{U}}_i^*, \widetilde{\mathrm{U}}_i} := -\mathcal{B}(\boldsymbol{\mathcal{U}}, \boldsymbol{\mathcal{V}}), \quad \forall \boldsymbol{\mathcal{U}}, \boldsymbol{\mathcal{V}} \in \widetilde{\mathrm{U}}_i.
\end{equation}
To cast the system into the abstract Cauchy framework within $\mathcal{H}_0$, it is necessary to restrict the unbounded operator $\mathcal{L}_{W,i}$ to a suitable domain. Thus, we define the domain $D(\mathcal{L}_{W,i})\subset\mathcal{H}_0$ as:
\begin{equation}
\label{eq:domain}
D(\mathcal{L}_{W,i}) := \{ \boldsymbol{\mathcal{U}} \in \widetilde{\mathrm{U}}_i : \exists \boldsymbol{\mathcal{F}}\in\mathcal{H}_0,\forall \boldsymbol{\mathcal{V}}\in\widetilde{\mathrm{U}}_{i}, \mathcal{B}(\boldsymbol{\mathcal{U}},\boldsymbol{\mathcal{V}})=\innerprod{\bm{\mathcal{F}}}{\bm{\mathcal{V}}}_{\widetilde{\mathrm{U}}_{i}^{*},\widetilde{\mathrm{U}}_{i}} \}.
\end{equation}

The abstract Cauchy problem can now be rigorously formulated as follows: find $\boldsymbol{\mathcal{U}}\in C([0,+\infty);D(\mathcal{L}_{W,i}))\cap C^1((0,+\infty);\mathcal{H}_0)$ such that 
\begin{equation}
\label{eq:homogeneous_R13_Cauchy}
\left\{
\begin{aligned} 
\frac{\dif}{\dif t}\bm{\mathcal{U}}(t)&= \mathcal{L}_{W,i} \bm{\mathcal{U}}(t), \quad t>0,\\ 
\bm{\mathcal{U}}(0) &= \bm{\mathcal{U}}_0.\\
\end{aligned}
\right.
\end{equation}
Here, the equation is understood as an equality in the Hilbert space $\mathcal{H}_0$. Since $\mathcal{H}_0$ is a Hilbert space, the Lumer-Phillips theorem dictates that the well-posedness of \eqref{eq:homogeneous_R13_Cauchy} reduces to verifying two properties for $\mathcal{L}_{W,i}$:
\begin{itemize}
        \item (Dissipativity): $\langle \mathcal{L}_{W,i}\bcal{U}, \bcal{U} \rangle_{\widetilde{\mathrm{U}}_i^*, \widetilde{\mathrm{U}}_i} \le 0$ for all $\bcal{U} \in D(\mathcal{L}_{W,i})$.
        \item (Range Condition): The operator $(\lambda_0 \mathcal{I} - \mathcal{L}_{W,i})$ is surjective from $D(\mathcal{L}_{W,i})$ to $\mathcal{H}_0$ for some $\lambda_0 > 0$, where $\mathcal{I}$ denotes the identity operator.
\end{itemize}
\begin{theorem}
    \label{Thm:dissipative}
    The differential operators $\mathcal{L}_{W,1}$, $\mathcal{L}_{W,2}$ are both dissipative.
\end{theorem}
\begin{proof}
    Let $i\in\{1,2\}$ and assume $\mathbf{g}_{\rm ext}=\mathbf{0}$. By the definition of $\mathcal{L}_{W,i}$, we have:
    $$\forall\bm{\mathcal{U}}\in D(\mathcal{L}_{W,i}),\quad\innerprod{\mathcal{L}_{W,i}\bm{\mathcal{U}}}{\bm{\mathcal{U}}}_{\tilde{\mathrm{U}}_{i}^{*},\tilde{\mathrm{U}}_i}=-\mathcal{B}(\bm{\mathcal{U}},\bm{\mathcal{U}}).$$
Since $\mathbf{g}_{\rm ext}=\mathbf{0}$, the boundary functional $\mathcal F_i$ vanishes. Using the energy identity~\eqref{eq:equivalence}, we have:
    \begin{equation}
        \mathcal{B}(\bm{\mathcal{U}},\bm{\mathcal{U}})=I_{\mathrm{bdry}}-\mathcal{W}_1.
    \end{equation}
The bulk entropy estimate in \cref{subsec:entropy} gives $\mathcal{W}_1\le 0$, while the homogeneous
Onsager boundary estimate in \cref{Lem:coefficients_on_bdry} gives $I_{\rm bdry}\ge 0$. Hence
$\mathcal B(\bm{\bcal{U}},\bm{\bcal{U}})\ge 0$, and therefore $\innerprod{\mathcal{L}_{W,i}\bm{\mathcal{U}}}{\bm{\mathcal{U}}}_{\tilde{\mathrm{U}}_{i}^{*},\tilde{\mathrm{U}}_i}\le 0$. Thus $\mathcal{L}_{W,i}$ is dissipative for $i=1,2$.
\end{proof}
\subsection{The range condition} Due to the distinct functional spaces $\widetilde{\mathrm{U}}_1$ and $\widetilde{\mathrm{U}}_2$, we must verify the range condition for the operators $\mathcal{L}_{W,1}$ and $\mathcal{L}_{W,2}$ independently.
\subsubsection{Non-Maxwell molecules} For non-Maxwell molecules, we decompose the state vector as $\bm{\mathcal{U}}=(\bm{\mathcal{S}},p)$, where $\bm{\mathcal{S}}\in\mathrm{T}_1$ and $p\coloneqq\rho+\theta\in L^2(\Omega)$. Notably, when evaluating the bilinear form at $(\bcal{U}, \bcal{U})$, the pressure-related cross terms cancel out exactly, yielding:
\begin{equation}
    \label{eq:bilinear_for_nM}
    \mathcal{B}(\bm{\mathcal{U}},\bm{\mathcal{U}})=\mathscr{A}_1(\bm{\mathcal{S}},\bm{\mathcal{S}})-g(\bm{\mathcal{S}},p)+g(\bm{\mathcal{S}},p)=\mathscr{A}_1(\bm{\mathcal{S}},\bm{\mathcal{S}}).
\end{equation}
This structural property directly yields the following lemma for the range condition.
\begin{theorem}
    \label{Lem:RC_of_nM}
     Under the strict constraints $3k_2^2 < k_1k_{10}$ and $25k_4^2 < 24k_3k_7$, the operator $\mathcal{L}_{W,1}$ satisfies the range condition.
\end{theorem}
\begin{proof}
 To establish the range condition, we must demonstrate that for any $\mathcal{F} \in \mathcal{H}_0$, there exists a state $\bm{\mathcal{U}}\in D(\mathcal{L}_{W,1})$ satisfying $(\lambda_0 \mathcal{I} - \mathcal{L}_{W,1})\bcal{U} = \bcal{F}$ for some $\lambda_0 > 0$. 
We cast this into its equivalent weak formulation: find $\bm{\mathcal{U}}\in\widetilde{\mathrm{U}}_{1}$ such that:
\begin{equation}
    \label{eq:weak_form}
    \mathcal{B}_{\lambda_0}(\bm{\mathcal{U}}, \bm{\mathcal{V}})\coloneqq \lambda_0 \langle \bcal{U}, \bcal{V} \rangle_{\mathcal{H}_0} +\mathcal{B}(\bcal{U},\bcal{V})=\innerprod{\bm{\mathcal{F}}}{\bcal{V}}_{\mathcal{H}_0},\quad\forall\bcal{V}\in\widetilde{\mathrm{U}}_1,
\end{equation}
then show the solution $\bcal{U}$ of \eqref{eq:weak_form} satisfies $\bcal{U}\in D(\mathcal{L}_{W,1})$.

Testing this formulation with $\bcal{V} = \bcal{U}$ and applying  \cref{Thm:Coercive_Part_1}, \cref{Thm:Coercive_Part_2} alongside \eqref{eq:bilinear_for_nM}, we obtain:
\begin{equation}
    \label{eq:coercive_for_Blambda}
    \begin{aligned}
    \mathcal{B}_{\lambda_0}(\bcal{U}, \bcal{U}) &= \lambda_0 \|\bcal{U}\|_{\mathcal{H}_0}^2 + \mathscr{A}_1(\bcal{S},\bcal{S}) \\
    &\gtrsim\lambda_0\|\bcal{U}\|_{\mathcal{H}_0}^2+\|\bar{\bm{s}}\|_{1}^2+\|\bm{u}\|_{1}^2+\|\bar{\bm{{\sigma}}}\|_{1}^2+\|\theta\|_{1}^2\gtrsim\|\bcal{U}\|_{\widetilde{\mathrm{U}}_1}^2.
    \end{aligned}
\end{equation}
This bound confirms that the continuous bilinear form $\mathcal{B}_{\lambda_0}$ is strictly coercive on $\widetilde{\mathrm{U}}_1$. Since $\mathcal{B}$ is continuous on $\widetilde{\mathrm{U}}_1$, by the Lax-Milgram theorem, there exists a unique weak solution $\bm{\mathcal{U}} \in \widetilde{\mathrm{U}}_1$ to the problem \eqref{eq:weak_form}.

Finally, rearranging the weak formulation gives:
\begin{equation}
    \label{eq:weak_solution_in_DL}
\mathcal{B}(\bm{\mathcal{U}}, \bm{\mathcal{V}}) = \langle \bm{\mathcal{F}} - \lambda_0\bm{\mathcal{U}}, \bm{\mathcal{V}}\rangle_{\mathcal{H}_0}, \quad \forall \bm{\mathcal{V}} \in \widetilde{\mathrm{U}}_1.
\end{equation}
Given that $\bm{\mathcal{U}},\bm{\mathcal{F}}\in\mathcal{H}_0$, it follows that the right-hand side representer $\bm{\mathcal{F}} - \lambda_0\bm{\mathcal{U}} \in \mathcal{H}_0$. Referring to the domain definition in \eqref{eq:domain}, this equality confirms that $\bm{\mathcal{U}} \in D(\mathcal{L}_{W,1})$. 
\end{proof}
\begin{remark}
    Compared to the steady-state formulation in \cref{sec:stablenM_wellposedness}, the space $\widetilde{\mathrm{U}}_1$ relaxes the zero-mean constraint for the pressure $p:= \rho + \theta$. The strictly positive resolvent parameter $\lambda_0 > 0$ compensates for this relaxation by inherently providing the requisite $L^2$-coercivity across the entire space.
\end{remark}
\subsubsection{Maxwell molecules} For the system of Maxwell molecules, the operator $\mathcal{B}_{\lambda_0}$ isn't strictly coercive on $\widetilde{\mathrm{U}}_2$. So we should apply the Banach-Nečas-Babuška theorem \cite{Alexandre2004PractiseFE} to examine the range condition for $\mathcal{L}_{W,2}$. 
\begin{theorem}
    \label{Lem:RC_of_Max}
    The operator $\mathcal{L}_{W,2}$ satisfies the range condition on $D(\mathcal{L}_{W,2})$.
\end{theorem}
\begin{proof}
Relaxing the zero-mean restriction on $p$, we define the unconstrained space $\mathrm{U}_2 := \mathrm{T}_2 \times \mathrm{W}$. Since $\widetilde{\mathrm{U}}_2\cong\mathrm{U}_2$, for any given $\bcal{F} \in \mathcal{H}_0$, the weak formulation of the resolvent equation $\lambda_0 \bcal{U} - \mathcal{L}_{W,2}\bcal{U} = \bcal{F}$ reads: find $\bcal{U} = (\bcal{S}, \bcal{W}) \in \mathrm{U}_2$ such that
\begin{equation}
    \mathcal{B}_{\lambda_0}(\bcal{U}, \bcal{T}) = \langle \bcal{F}, \bcal{T} \rangle_{\mathcal{H}_0}, \quad \forall\bcal{T}=(\bcal{R},\bcal{V}) \in \mathrm{U}_2,
\end{equation}
where the full resolvent bilinear form is given by:
\begin{equation}
    \mathcal{B}_{\lambda_0}(\bcal{U}, \bcal{T}) := \lambda_0 \innerprod{\bcal{U}}{\bcal{T}}_{\mathcal{H}_0} + \mathscr{A}_2(\bcal{S}, \bcal{R}) + \mathscr{B}_2(\bcal{R}, \bcal{W}) - \mathscr{B}_2(\bcal{S}, \bcal{V}).
\end{equation}
Since $\mathscr{B}_{\lambda_0}$ is a continuous bilinear form, applying the Banach-Nečas-Babuška (BNB) theorem requires establishing the global inf-sup condition (weak coercivity) over $\mathrm{U}_2$:
\begin{equation}
    \label{eq:global_infsup}
    \inf_{\bcal{U}\in\mathrm{U}_2\setminus\{ \mathbf{0}\}}\sup_{\bcal{T}\in\mathrm{U}_2\setminus\{\mathbf{0}\}}\frac{\mathcal{B}_{\lambda_0}(\bcal{U},\bcal{T})}{\|\bcal{U}\|_{\mathrm{U}_2}\|\bcal{T}\|_{\mathrm{U}_2}}\ge C_{0}>0.
\end{equation}
For any given non-zero $\bcal{U}\coloneqq(\bcal{S},\bcal{W})\in\mathrm{U}_2$, Theorem \ref{thm:inf_sup_modified_Maxwell} guarantees the existence of a specific $\bcal{R}_{\bcal{W}}\in\mathrm{T}_2$ with $\|\bcal{R}_{\bcal{W}}\|_{\mathrm{T}_2}=\|\bcal{W}\|_{\mathrm{W}}$ such that
\begin{equation}
\label{eq:infsup_of_B}
    \exists \alpha_0>0,\quad|\mathscr{B}_2(\bcal{R}_{\bcal{W}},\bcal{W})|= \alpha_{0}\|\bcal{R}_{\bcal{W}}\|_{\mathrm{T}_2}\|\bcal{W}\|_{\mathrm{W}}.
\end{equation}
We explicitly construct the test function $\bcal{T} := (\bcal{S} + \delta \bcal{R}_{\bcal{W}}, \bcal{W}) \in \mathrm{U}_2$, where $\delta > 0$ is an undetermined parameter. Substituting $\bcal{T}$ into the bilinear form, we can split the expression into the diagonal evaluation and a perturbation term:
\begin{equation}
\mathcal{B}_{\lambda_0}(\bcal{U}, \bcal{T}) = \mathcal{B}_{\lambda_0}(\bcal{U}, \bcal{U}) + \mathcal{B}_{\lambda_0}(\bcal{U}, (\delta\bcal{R}_{\bcal{W}}, \boldsymbol{0})).
\end{equation}
To rigorously control the perturbation term, we isolate the symmetric part of $\mathcal{B}_{\lambda_0}$. We define the symmetric, positive semi-definite bilinear form $\mathcal{C}_{\lambda_0}:\mathrm{U}_2\times\mathrm{U}_2\rightarrow\mathbb{R}$ as:
\begin{equation}
    \mathcal{C}_{\lambda_{0}}(\bcal{U},\bcal{T})\coloneqq \lambda_0 \innerprod{\bcal{U}}{\bcal{T}}_{\mathcal{H}_0} + \mathscr{A}_2(\bcal{S}, \bcal{R}).
\end{equation}
Observe that the diagonal evaluation satisfies $\mathcal{C}_{\lambda_0}(\bcal{U}, \bcal{U}) = \mathcal{B}_{\lambda_0}(\bcal{U}, \bcal{U}) \gtrsim \|\bcal{S}\|_{\mathrm{T}_2}^2 + \|\boldsymbol{u}\|_0^2 + \|\theta\|_0^2$. Using this symmetric form, the perturbation term expands to:
\begin{equation}
\mathcal{B}_{\lambda_0}(\bcal{U}, (\delta\bcal{R}_{\bcal{W}}, \boldsymbol{0})) = \mathcal{C}_{\lambda_0}(\bcal{U}, (\delta\bcal{R}_{\bcal{W}}, \boldsymbol{0})) + \delta\mathscr{B}_2(\bcal{R}_{\bcal{W}}, \bcal{W}).
\end{equation}
Since $\mathcal{C}_{\lambda_0}$ is positive semi-definite, it admits the generalized Cauchy-Schwarz inequality. Combining this with Young's inequality, we bound the symmetric cross term
\begin{equation}
        \mathcal{C}_{\lambda_0}(\bcal{U},(\delta\bcal{R}_{\bcal{W}},0))
        \ge - \tfrac{1}{2} \mathcal{C}_{\lambda_0}(\bcal{U},\bcal{U})-2\delta^2\mathcal{C}_{\lambda_0}((\bcal{R}_{\bcal{W}},0),(\bcal{R}_{\bcal{W}},0)).
\end{equation}
Furthermore, by the continuity of $\mathscr{A}_2$ and the inner product, there exists a generic constant $C_1 > 0$ such that $\mathcal{C}_{\lambda_0}((\bcal{R}_{\bcal{W}}, \boldsymbol{0}), (\bcal{R}_{\bcal{W}}, \boldsymbol{0})) \le C_1\|\bcal{R}_{\bcal{W}}\|_{\mathrm{T}_2}^2$. Combining these estimates and substituting them back into the expansion of $\mathscr{B}_{\lambda_0}(\mathcal{U}, \mathcal{T})$, we obtain:
\begin{equation}
    \begin{aligned}
\mathcal{B}_{\lambda_0}(\bcal{U}, \bcal{T}) &= \mathcal{B}_{\lambda_0}(\bcal{U}, \bcal{U}) + \mathcal{C}_{\lambda_0}(\bcal{U}, (\delta\bcal{R}_{\bcal{W}}, \boldsymbol{0})) + \delta\mathscr{B}_2(\bcal{R}_{\bcal{W}}, \bcal{W}) \\
&\ge \tfrac{1}{2}\mathcal{B}_{\lambda_0}(\bcal{U}, \bcal{U}) + \delta\mathscr{B}_2(\bcal{R}_{\bcal{W}}, \bcal{W}) - 2C_1\delta^2\|\bcal{R}_{\bcal{W}}\|_{\mathrm{T}_2}^2.
\end{aligned}
\end{equation}
Recalling from \cref{thm:inf_sup_modified_Maxwell} that $\mathscr{B}_2(\bcal{R}_{\bcal{W}}, \bcal{W}) \ge \alpha_0\|\bcal{W}\|_{\mathrm{W}}^2$ and $\|\bcal{R}_{\bcal{W}}\|_{\mathrm{T}_2} = \|\bcal{W}\|_{\mathrm{W}}$, the inequality simplifies directly to:
\begin{equation}
    \mathcal{B}_{\lambda_0}(\bcal{U}, \bcal{T}) \gtrsim \|\bcal{S}\|_{\mathrm{T}_2}^2 + (\delta\alpha_0 - 2C_1\delta^2)\|\bcal{W}\|_{\mathrm{W}}^2.
\end{equation}

By choosing $\delta > 0$ sufficiently small such that $\delta\alpha_0 - 2C_1\delta^2 > 0$, we ensure $\mathcal{B}_{\lambda_0}(\bcal{U}, \bcal{T}) \gtrsim \|\bcal{S}\|_{\mathrm{T}_2}^2 + \|\bcal{W}\|_{\mathrm{W}}^2 = \|\bcal{U}\|_{\mathrm{U}_2}^2$. Furthermore, the triangle inequality directly bounds the test function: $\|\bcal{T}\|_{\mathrm{U}_2} \le \|\bcal{S}\|_{\mathrm{T}_2} + \delta\|\bcal{R}_{\bcal{W}}\|_{\mathrm{T}_2} + \|\bcal{W}\|_{\mathrm{W}} \lesssim \|\bcal{U}\|_{\mathrm{U}_2}$. Combining these inequalities immediately yields the required inf-sup condition.

Since $\mathcal{B}_{\lambda_0}(\bcal{U}, \bcal{U}) > 0$ for all non-zero $\bcal{U} \in \mathrm{U}_2$, injectivity is inherently satisfied. Thus, the BNB theorem applies, establishing that $\lambda_0\mathcal{I} - \mathcal{L}_{\text{W},2}$ is surjective from $\mathrm{U}_2$ to $\mathcal{H}_0$. Following the same algebraic deduction as in (5.13), this surjectivity seamlessly restricts to $D(\mathcal{L}_{\text{W},2})$, concluding the proof.
\end{proof}

From \cref{Thm:dissipative,Lem:RC_of_nM,Lem:RC_of_Max}, we obtain the final conclusion of this section.
\begin{theorem}[Global well-posedness of the time-dependent problem]
\label{thm:time-dependent-main}
For \(i=1,2\), the operator \(\mathcal L_{W,i}\) generates a strongly continuous
contraction semigroup \(\mathcal S_i(t)\) on \(\mathcal{H}_0\). Consequently, for any
initial data \(\bcal U_0\in D(\mathcal{L}_{W,i})\), the homogeneous abstract Cauchy
problem~\eqref{eq:homogeneous_R13_Cauchy} admits a unique classical solution
\begin{equation}
    \bcal U(t)=\mathcal S_i(t)\bcal U_0
    \in C([0,\infty);D(\mathcal L_{W,i}))
    \cap C^1((0,\infty);\mathcal H_0).
\end{equation}
Moreover, the solution obeys the contraction estimate $\|\bcal{U}(t)\|_{\mathcal{H}_0}\le\|\bcal{U}_0\|_{\mathcal{H}_0}$ for $t>0$.
\end{theorem}

The above theorem completes the proof of global well-posedness for the
homogeneous time-dependent R13 system in both the non-Maxwell and Maxwell
regimes.


\begin{remark}[Non-homogeneous boundary data]
The semigroup analysis above is formulated for homogeneous Onsager boundary data. For sufficiently regular non-homogeneous boundary data satisfying the
necessary compatibility conditions, the problem can be reduced to the homogeneous case by a standard lifting argument. More precisely, if there exists a lifting
$\bcal{U}_W(t)\in\tilde{\mathrm{U}}_i$ such that $\mathcal{G}(\bcal{U}_W(t))=\mathbf{g}_{\rm ext}(t)$, then the shifted
unknown $\widehat{\bcal{U}}=\bcal{U}-\bcal{U}_W$ satisfies an inhomogeneous abstract Cauchy problem with homogeneous boundary data.
The corresponding mild solution is then obtained
by Duhamel's formula. Since the construction of such liftings depends on the
regularity of the prescribed wall data, we 
do not elaborate on this extension here.
\end{remark}

\section{Conclusion and future work}
\label{Sec:Conclusion}
In this paper, we establish the global well-posedness of the linearized R13 equations in a weak framework. By exploiting the macroscopic-microscopic correspondence of the boundary fluxes, we first derive entropy inequalities for the R13 system on bounded Lipschitz domains. For the steady-state problem, we prove the well-posedness for non-Maxwell molecules under the stated constraints by establishing a novel boundary-related Korn-type inequality. For Maxwell molecules, we introduce a tailored functional framework, different from that in \cite{lewintan2025wellposedness}, to recover the well-posedness of the system. Finally, by combining semigroup arguments with the Lumer-Phillips theorem, we extend the analysis to the time-dependent R13 system.

The results obtained here provide a rigorous analytical foundation for the numerical analysis of the R13 system. In particular, the functional framework proposed for Maxwell molecules is better aligned with the design of mixed finite element methods and helps address the issue of mismatched convergence orders in existing formulations. Moreover, by recasting the continuous problem into a robust saddle-point framework, the construction of stable numerical schemes is naturally linked to the verification of suitable discrete inf-sup conditions.

Several directions deserve further investigation. While the present work focuses on monatomic gases, polyatomic gases are more common in practical applications. It is therefore natural to extend the present analysis to stable linearized R13 models for polyatomic gases. In addition, guided by the continuous functional settings established in this paper, a natural next step is to construct, analyze, and implement efficient and stable mixed finite element schemes for the linearized R13 system.

\appendix
\section{Thermodynamic constraints for coefficients}
\label{appA:Coefficients_for_k}
In this appendix, we explicitly detail the dependence of the coefficient sets $(k_1, k_2, k_{10})$ and $(k_3, k_4, k_7)$ on the molecular interaction parameter $\eta$, and numerically verify the validity of the strict constraints.

To facilitate the presentation, we introduce the dimensionless ratios $(w_1, w_2) := \left( \frac{3k_2^2}{k_1 k_{10}}, \frac{25k_4^2}{24k_3 k_7} \right)$ and the corresponding discriminants $(z_1, z_2) := (k_1 k_{10} - 3k_2^2, 24k_3 k_7 - 25k_4^2)$. Under this notation, the strict constraints are expressed as $w_1, w_2 < 1$ and $z_1, z_2 > 0$. \cref{Tab:strict_thermodynamic} demonstrates that the strict thermodynamic constraints are numerically satisfied within our R13 framework.
\begin{table}[htbp]
\label{Tab:strict_thermodynamic}
\begin{tabular}{|c|ccccc|}
\hline
$\eta$   & $k_1$                  & $k_2$                  & $k_{10}$               & $z_1$                   & $w_1$  \\
7        & $3.0773\times 10^{-3}$ & $1.2550\times 10^{-5}$ & $2.8590\times 10^{-7}$ & $4.0729\times 10^{-10}$ & 0.5371 \\
10       & $8.7436\times 10^{-3}$ & $4.5818\times 10^{-5}$ & $1.1896\times 10^{-6}$ & $4.1035\times 10^{-9}$  & 0.6055 \\
17       & $1.6341\times 10^{-2}$ & $1.0021\times 10^{-4}$ & $2.8475\times 10^{-6}$ & $1.5190\times 10^{-8}$  & 0.6474 \\
$\infty$ & $3.0261\times 10^{-2}$ & $2.0798\times 10^{-4}$ & $6.3621\times 10^{-6}$ & $6.2756\times 10^{-8}$  & 0.6740 \\ \hline
$\eta$   & $k_3$                  & $k_4$                  & $k_7$                  & $z_2$                   & $w_2$  \\
7        & $2.6072\times 10^{-3}$ & $4.8885\times 10^{-2}$ & $9.7119\times 10^{-1}$ & $1.0265\times 10^{-3}$  & 0.9831 \\
10       & $7.4080\times 10^{-3}$ & $8.1805\times 10^{-2}$ & $9.5624\times 10^{-1}$ & $2.7104\times 10^{-3}$  & 0.9841 \\
17       & $1.3840\times 10^{-2}$ & $1.1124\times 10^{-1}$ & $9.4576\times 10^{-1}$ & $4.7852\times 10^{-3}$  & 0.9848 \\
$\infty$ & $2.5607\times 10^{-2}$ & $1.5056\times 10^{-1}$ & $9.3584\times 10^{-1}$ & $8.4295\times 10^{-3}$  & 0.9853 \\ \hline
\end{tabular}
\caption{Numerical evaluation of the thermodynamic constraints for non-Maxwell molecules.}
\end{table}

\section*{Acknowledgments}
The authors would like to thank Huiteng Li (KAUST) for his valuable discussions regarding Korn-type inequalities and semigroup theory.

\bibliographystyle{siamplain}
\bibliography{references}
\end{document}


\maketitle
\allowdisplaybreaks[1]

\section{Weak form}
This section presents the detailed derivation of all coefficients referenced in Section 4.1. We restate the Onsager boundary conditions with numbered labels for precise reference:
\begin{align}
    &u_n = 0, \label{eq:BC_1} \\
    &\bar{s}_{n} = \tilde{\chi} \Big[ m_{11} (\theta -\theta^{W}) \label{eq:BC_2} \\
    & \qquad +m_{12} \bar{\sigma}_{nn} - \Kn \left(m_{13}\pdfFrac{\bar{s}_j}{x_j}+m_{14}\pdfFrac{\bar{s}_{\langle n}}{x_{n \rangle}}+ m_{15}\pdfFrac{u_{\langle n}}{x_{n \rangle}}\right) \Big], \notag \\ 
    &m_{26}\bar{s}_{n} + m_{27} \Kn \pdfFrac{\theta}{x_n} - m_{28} \Kn \pdfFrac{\bar{\sigma}_{nj}}{x_j} = \tilde{\chi} \Big[-m_{21} (\theta -\theta^{W}) \label{eq:BC_3} \\
    & \qquad + m_{22} \bar{\sigma}_{nn} +  m_{23} \Kn \pdfFrac{\bar{s}_j}{x_j} +  m_{24} \Kn \pdfFrac{\bar{s}_{\langle n}}{x_{n \rangle}}  + m_{25} \Kn \pdfFrac{u_{\langle n}}{x_{n \rangle}} \Big], \notag \\ 
    &\bar{\sigma}_{t_i n} = \tilde{\chi} \big[m_{31} (u_{t_i} -u^{W}_{t_i}) \label{eq:BC_4} \\
    & \qquad +m_{32}\bar{s}_{t_i} - \Kn\left(m_{33}\pdfFrac{\bar{\sigma}_{t_i j}}{x_j}  + m_{34}\pdfFrac{\bar{\sigma}_{\langle t_i n}}{x_{n \rangle}}  - m_{35}\pdfFrac{\theta}{x_{t_i}}\right) \big],   \notag \\ 
    &m_{46}\bar{\sigma}_{t_in} + m_{47}\Kn\pdfFrac{u_{\langle t_i}}{x_{n\rangle}} + m_{48}\Kn\pdfFrac{\bar{s}_{\langle t_i}}{x_{n\rangle}} = - \tilde{\chi} \big[-m_{41} (u_{t_i} -u^{W}_{t_i})  \label{eq:BC_5} \\
    & \qquad +m_{42}\bar{s}_{t_i} + m_{43} \Kn \pdfFrac{\bar{\sigma}_{t_i j}}{x_j}  + m_{44} \Kn  \pdfFrac{\bar{\sigma}_{\langle t_i n}}{x_{n \rangle}}  - m_{45} \Kn \pdfFrac{\theta}{x_{t_i}} \big], \notag \\ 
    &m_{56}\bar{\sigma}_{t_in} + m_{57}\Kn\pdfFrac{u_{\langle t_i}}{x_{n\rangle}} + m_{58}\Kn\pdfFrac{\bar{s}_{\langle t_i}}{x_{n\rangle}} = - \tilde{\chi} \big[-m_{51} (u_{t_i} -u^{W}_{t_i})  \label{eq:BC_6} \\
    & \qquad +m_{52}\bar{s}_{t_i} + m_{53} \Kn  \pdfFrac{\bar{\sigma}_{t_i j}}{x_j}  + m_{54} \Kn  \pdfFrac{\bar{\sigma}_{\langle t_i n}}{x_{n \rangle}}  - m_{55} \Kn \pdfFrac{\theta}{x_{t_i}} \big], \notag \\ 
    &m_{66} \bar{s}_n + \Kn\left(m_{67} \pdfFrac{\bar{\sigma}_{\langle nn}}{x_{n \rangle}} +m_{68}\pdfFrac{\bar{\sigma}_{nj }}{x_{j}} + m_{69}\pdfFrac{\theta}{x_n} \right)  = - \tilde{\chi} \big[-m_{61} (\theta -\theta^{W}) \label{eq:BC_7} \\
    & \qquad +m_{62} \bar{\sigma}_{nn} + m_{63} \Kn \pdfFrac{\bar{s}_j}{x_j} +  m_{64} \Kn \pdfFrac{\bar{s}_{\langle n}}{x_{n \rangle}} + m_{65} \Kn \pdfFrac{u_{\langle n}}{x_{n \rangle}}\big], \notag \\ 
    &\Kn\left(\pdfFrac{\bar{\sigma}_{\langle t_i t_i}}{x_{n \rangle}} + \frac{1}{2} \pdfFrac{\bar{\sigma}_{\langle n n}}{x_{n \rangle}} \right) =  -\tilde{\chi} m_{71} \left( \bar{\sigma}_{t_i t_i}  + \frac{1}{2} \bar{\sigma}_{nn}  \right), \label{eq:BC_8} \\
    &\Kn \pdfFrac{\bar{\sigma}_{\langle t_1 t_2}}{x_{n\rangle}}  = -\tilde{\chi} m_{81} \bar{\sigma}_{t_1 t_2}. \label{eq:BC_9}
\end{align}
Einstein summation convention is adopted for the repeated index $j$, and all conditions hold for $i \in \{1,2\}$.

We first define the proportionality coefficients via the matching of derivative terms between the governing equations and boundary conditions. These coefficients are uniquely determined by the parameters $k_i$ and $m_{jk}$:
\begin{align}
    &C_1 = \frac{\frac{5}{2}k_4}{m_{15}} = \frac{2k_6}{m_{13}} = \frac{\frac{12}{5}k_7}{m_{14}}, \label{eq:C1_def} \\
    &C_2 = \frac{m_{43}}{m_{33}} = \frac{m_{44}}{m_{34}} = \frac{m_{45}}{m_{35}}, \quad C_3 = \frac{m_{53}}{m_{33}} = \frac{m_{54}}{m_{34}} = \frac{m_{55}}{m_{35}}, \label{eq:C7_def} \\
    &C_4 = \frac{m_{23}}{m_{13}} = \frac{m_{24}}{m_{14}} = \frac{m_{25}}{m_{15}}, \label{eq:C4_def} \\
    &C_5 = \frac{k_1}{m_{27}} = \frac{k_2}{m_{28}}, \label{eq:C5_def} \\
    &C_6 = \frac{3k_2}{m_{35}} = \frac{4k_9}{m_{34}} = \frac{k_{10}}{m_{33}}, \label{eq:C6_def} \\
    & C_7 = \frac{m_{63}}{m_{13}} = \frac{m_{64}}{m_{14}} = \frac{m_{65}}{m_{15}}. \label{eq:c2_def}
\end{align}
In the following derivation, some boundary coefficients may appear through different component equations. \cref{rem:consist_of_paras} guarantees that these alternative expressions are consistent, so the weak formulation is well defined.

\subsection{Heat flux balance}
Recall that the strong form of the heat flux balance equation reads:
\begin{equation*}
    \partial_t \qbarb = -\frac{5}{2}k_0\nabla \theta + \frac{5}{2}k_4 \Kn \nabla \cdot (\nabla \bm{u})_{\stf} + 2k_6 \Kn \nabla(\nabla \cdot \qbarb) + \frac{12}{5} k_7 \Kn \nabla \cdot (\nabla \qbarb)_{\stf} - k_8 \nabla \cdot \sigmabarb - \frac{2l_1}{3\Kn}\qbarb.
\end{equation*}
Multiplying both sides by the vector test function $\frac{2}{5}\bar{\bm{r}}$, integrating over the domain $\Omega$, and applying integration by parts yields:
\begin{equation}
    \begin{array}{ll}
        &\quad\frac{2}{5}\smallip{\partial_t \qbarb}{\bar{\bm{r}}}
        = k_0 \smallip{\theta}{\nabla\cdot\bar{\bm{r}}} - k_0 \smallip{\theta}{\bar{r}_n}_b \\
        &- k_4 \Kn \smallip{(\nabla \bm{u})_{\stf}}{(\nabla \bar{\bm{r}})_{\stf}} + k_4 \Kn  \smallip{\pdfFrac{u_{\langle n}}{x_{n\rangle}}}{\bar{r}_{n}}_b + k_4 \Kn \sum_{i=1}^2 \smallip{\pdfFrac{u_{\langle t_i}}{x_{n\rangle}}}{\bar{r}_{t_i}}_b \\
        & - \frac{4}{5}k_6 \Kn \smallip{\nabla\cdot\qbarb}{\nabla\cdot\bar{\bm{r}}} + \frac{4}{5}k_6 \Kn \smallip{\nabla\cdot\qbarb}{\bar{r}_n}_b \\
        & - \frac{24}{25}k_7 \Kn \smallip{(\nabla \qbarb)_{\stf}}{(\nabla \bar{\bm{r}})_{\stf}} + \frac{24}{25}k_7 \Kn \smallip{\pdfFrac{\bar{s}_{\langle n}}{x_{n\rangle}}}{\bar{r}_{n}}_b + \frac{24}{25}k_7 \Kn \sum_{i=1}^2 \smallip{\pdfFrac{\bar{s}_{\langle t_i}}{x_{n\rangle}}}{\bar{r}_{t_i}}_b \\
        &- \frac{2}{5}k_8 \smallip{\nabla \cdot \sigmabarb}{\bar{\bm{r}}} - \frac{4l_1}{15\Kn} \smallip{\qbarb}{\bar{\bm{r}}}.
    \end{array}
    \label{eq:heat_weak_scaled}
\end{equation}

We now substitute the boundary conditions to eliminate the normal derivatives appearing in the boundary integral terms.

\subsubsection{Normal component boundary terms}
Using \eqref{eq:C1_def}, we rewrite the normal derivative terms as:
\begin{equation}
    \frac{4}{5}k_6 \Kn \nabla\cdot\qbarb + \frac{24}{25}k_7 \Kn \pdfFrac{\bar{s}_{\langle n}}{x_{n \rangle}} + k_4 \Kn \pdfFrac{u_{\langle n}}{x_{n \rangle}} = \frac{2}{5} C_1 \Kn \left( m_{13}\nabla\cdot\qbarb + m_{14}\pdfFrac{\bar{s}_{\langle n}}{x_{n\rangle}} + m_{15}\pdfFrac{u_{\langle n}}{x_{n \rangle}} \right).
\end{equation}
Rearranging the derivative terms in \eqref{eq:BC_2} gives:
\begin{equation}\label{eq:m1eq_deri_left}
    \Kn \left( m_{13}\nabla\cdot\qbarb + m_{14}\pdfFrac{\bar{s}_{\langle n}}{x_{n \rangle}} + m_{15}\pdfFrac{u_{\langle n}}{x_{n \rangle}} \right) = m_{11}(\theta-\theta^W) + m_{12}\bar{\sigma}_{nn} - \frac{1}{\tilde{\chi}} \bar{s}_n.
\end{equation}
Substituting this result back into the normal boundary integrals of \eqref{eq:heat_weak_scaled} and collecting the coefficients for $\smallip{\bar{s}_n}{\bar{r}_n}_b$, $\smallip{\theta}{\bar{r}_n}_b$, and $\smallip{\bar{\sigma}_{nn}}{\bar{r}_n}_b$, we obtain:
\begin{displaymath}
    S_5 = \frac{2 C_1}{5 \tilde{\chi}}, \qquad
    R_1 = \frac{2}{5}C_1 m_{11} - k_0, \qquad
    R_3 = \frac{2}{5} C_1 m_{12} .
\end{displaymath}
The wall temperature term contributes $-\frac{2}{5}C_1 m_{11} \smallip{\theta^W}{\bar{r}_n}_b$ to the linear functional $L_1(\bar{\bm{r}})$.

\subsubsection{Tangential component boundary terms}
We isolate from \eqref{eq:heat_weak_scaled} all boundary terms involving the tangential components $\bar{r}_{t_i}$ of the test function:
\begin{equation}
    \mathcal{B}_{\text{heat}}^t \coloneqq \sum_{i=1}^2 \left( k_4 \Kn \smallip{\pdfFrac{u_{\langle t_i}}{x_{n\rangle}}}{\bar{r}_{t_i}}_b + \frac{24}{25}k_7 \Kn \smallip{\pdfFrac{\bar{s}_{\langle t_i}}{x_{n\rangle}}}{\bar{r}_{t_i}}_b \right).
    \label{eq:heat_tangent_boundary_raw}
\end{equation}
We will use boundary conditions \eqref{eq:BC_4}, \eqref{eq:BC_5}, and \eqref{eq:BC_6} to eliminate the normal derivatives in \eqref{eq:heat_tangent_boundary_raw}.

First, we rearrange these boundary conditions as follows:
\begin{align}
    &m_{33} \Kn \pdfFrac{\bar{\sigma}_{t_i j}}{x_j} + m_{34} \Kn \pdfFrac{\bar{\sigma}_{\langle t_i n}}{x_{n\rangle}} - m_{35} \Kn \pdfFrac{\theta}{x_{t_i}} = m_{31}(u_{t_i} - u^W_{t_i}) + m_{32} \bar{s}_{t_i} - \frac{1}{\tilde{\chi}} \bar{\sigma}_{nt_i}, \label{eq:BC_T1} \\
    &m_{43} \Kn \pdfFrac{\bar{\sigma}_{t_i j}}{x_j} + m_{44} \Kn \pdfFrac{\bar{\sigma}_{\langle t_i n}}{x_{n\rangle}} - m_{45} \Kn \pdfFrac{\theta}{x_{t_i}} + \frac{1}{\tilde{\chi}} m_{47} \Kn \pdfFrac{u_{\langle t_i}}{x_{n\rangle}} + \frac{1}{\tilde{\chi}} m_{48} \Kn \pdfFrac{\bar{s}_{\langle t_i}}{x_{n\rangle}} \label{eq:BC_T2} \\
    &\quad = m_{41}(u_{t_i} - u^W_{t_i}) - m_{42} \bar{s}_{t_i} - \frac{1}{\tilde{\chi}} m_{46} \bar{\sigma}_{nt_i}, \nonumber \\
    &m_{53} \Kn \pdfFrac{\bar{\sigma}_{t_i j}}{x_j} + m_{54} \Kn \pdfFrac{\bar{\sigma}_{\langle t_i n}}{x_{n\rangle}} - m_{55} \Kn \pdfFrac{\theta}{x_{t_i}} + \frac{1}{\tilde{\chi}} m_{57} \Kn \pdfFrac{u_{\langle t_i}}{x_{n\rangle}} + \frac{1}{\tilde{\chi}} m_{58} \Kn \pdfFrac{\bar{s}_{\langle t_i}}{x_{n\rangle}} \label{eq:BC_T3} \\
    &\quad = m_{51}(u_{t_i} - u^W_{t_i}) - m_{52} \bar{s}_{t_i} - \frac{1}{\tilde{\chi}} m_{56} \bar{\sigma}_{nt_i}. \nonumber
\end{align}
Using the proportionality coefficients in \eqref{eq:C7_def}, we eliminate the redundant derivative terms (the first three terms on the left-hand side) via linear combinations:
\begin{align}
    (\eqref{eq:BC_T2} - C_2 \cdot \eqref{eq:BC_T1})\cdot \tilde{\chi}&= m_{47} \Kn \pdfFrac{u_{\langle t_i}}{x_{n\rangle}} + m_{48} \Kn \pdfFrac{\bar{s}_{\langle t_i}}{x_{n\rangle}} \label{eq:BC_T4} \\
    &= A_1 (u_{t_i} - u^W_{t_i}) + A_2 \bar{s}_{t_i} + A_3 \bar{\sigma}_{nt_i},  \notag \\
    (\eqref{eq:BC_T3} - C_3 \cdot \eqref{eq:BC_T1})\cdot \tilde{\chi}&= m_{57} \Kn \pdfFrac{u_{\langle t_i}}{x_{n\rangle}} + m_{58} \Kn \pdfFrac{\bar{s}_{\langle t_i}}{x_{n\rangle}}\label{eq:BC_T5}\\
    &= B_1 (u_{t_i} - u^W_{t_i}) + B_2 \bar{s}_{t_i} + B_3 \bar{\sigma}_{nt_i},  \notag
\end{align}
where the intermediate coefficients are defined as:
\begin{equation}\label{eq:A13B13}
    \begin{split}
    &A_1 =\tilde{\chi} ( m_{41} - C_2 m_{31} ), \quad A_2 = \tilde{\chi} (-m_{42} - C_2 m_{32} ), \quad A_3 = -m_{46} + C_2, \\
    &B_1 =\tilde{\chi} ( m_{51} - C_3 m_{31} ), \quad B_2 = \tilde{\chi} (-m_{52} - C_3 m_{32} ), \quad B_3 = -m_{56} + C_3.
    \end{split}
\end{equation}

We then introduce linear combination coefficients to match the derivative coefficients $k_4$ and $\frac{24}{25}k_7$ in \eqref{eq:heat_tangent_boundary_raw}:
\begin{equation*}
    \begin{pmatrix} \alpha_1 \\ \beta_1 \end{pmatrix} = \begin{pmatrix} m_{47} & m_{57} \\ m_{48} & m_{58} \end{pmatrix}^{-1} \begin{pmatrix} k_4 \\ \frac{24}{25}k_7 \end{pmatrix},
\end{equation*}
which satisfy the matching conditions:
\begin{equation*}
    \alpha_1 m_{47} + \beta_1 m_{57} = k_4, \quad \alpha_1 m_{48} + \beta_1 m_{58} = \frac{24}{25}k_7.
\end{equation*}
Therefore, taking the linear combination $\alpha_1 \cdot \eqref{eq:BC_T4} +  \beta_1 \cdot \eqref{eq:BC_T5}$ gives:
\begin{equation*}
    k_4 \Kn \pdfFrac{u_{\langle t_i}}{x_{n\rangle}} + \frac{24}{25}k_7 \Kn \pdfFrac{\bar{s}_{\langle t_i}}{x_{n\rangle}} = Z_1 (u_{t_i} - u^W_{t_i}) + Z_2 \bar{s}_{t_i} + Z_3 \bar{\sigma}_{nt_i},
\end{equation*}
where
\begin{align*}
    Z_1 = \alpha_1 A_1 + \beta_1 B_1, \quad Z_2 = \alpha_1 A_2 + \beta_1 B_2, \quad Z_3 = \alpha_1 A_3 + \beta_1 B_3.
\end{align*}

Substituting this result back into the tangential boundary integral \eqref{eq:heat_tangent_boundary_raw} gives:
\begin{equation}
    \mathcal{B}_{\text{heat}}^t = \sum_{i=1}^2 \smallip{Z_1 (u_{t_i} - u^W_{t_i}) + Z_2 \bar{s}_{t_i} + Z_3 \bar{\sigma}_{nt_i} }{\bar{r}_{t_i}}_b ,
\end{equation}
which yields the explicit tangential boundary coefficients:
\begin{align*}
    &S_1 = -Z_2 = -\left( \alpha_1 A_2 + \beta_1 B_2 \right), \qquad T_1 = -Z_1 = -\left( \alpha_1 A_1 + \beta_1 B_1 \right), \\
    &R_2 = Z_3 = \left( \alpha_1 A_3 + \beta_1 B_3 \right).
\end{align*}
The wall velocity term contributes $T_1 \smallip{u_{t_i}^W}{\bar{r}_{t_i}}_b$ to the linear functional $L_1(\bar{\bm{r}})$.

\subsection{Energy balance}
The strong form of the energy balance equation reads:
\begin{equation*}
    \partial_t \theta = -\frac{2}{3}\nabla \cdot \bm{u} - \frac{2}{3}k_0\nabla \cdot \qbarb + k_1 \Kn \Delta \theta - k_2 \Kn \nabla \cdot (\nabla \cdot \sigmabarb).
\end{equation*}
Multiplying both sides by the scalar test function $\frac{3}{2}\gamma$, integrating over $\Omega$, and applying integration by parts gives:
\begin{equation}
    \begin{array}{rl}
        \frac{3}{2} \smallip{\partial_t \theta}{\gamma}
        &=-\smallip{\nabla \cdot \bm{u}}{\gamma} - k_0 \smallip{\nabla \cdot \qbarb}{\gamma} \\
        &\quad - \frac{3}{2}k_1 \Kn \smallip{\nabla\theta}{\nabla\gamma} + \frac{3}{2}k_1 \Kn \smallip{\pdfFrac{\theta}{x_n}}{\gamma}_b \\
        &\quad + \frac{3}{2}k_2 \Kn \smallip{\nabla\cdot\sigmabarb}{\nabla\gamma} - \frac{3}{2}k_2 \Kn \smallip{\pdfFrac{\bar{\sigma}_{nj}}{x_j}}{\gamma}_b,
    \end{array}
    \label{eq:energy_weak_scaled}
\end{equation}
Using \eqref{eq:C5_def}, we rewrite the derivative terms in boundary integrals as:
\begin{equation*}
    \frac{3}{2}k_1 \Kn \pdfFrac{\theta}{x_n} - \frac{3}{2}k_2 \Kn \pdfFrac{\bar{\sigma}_{nj}}{x_j} = \frac{3}{2}C_5 \left( m_{27} \Kn \pdfFrac{\theta}{x_n} - m_{28} \Kn \pdfFrac{\bar{\sigma}_{nj}}{x_j} \right).
\end{equation*}
From boundary condition \eqref{eq:BC_3}, we rearrange terms to get:
\begin{equation*}
    \begin{aligned}
	&m_{27}\Kn\pdfFrac{\theta}{x_n} - m_{28}\Kn\pdfFrac{\bar{\sigma}_{nj}}{x_j}\\
	= &\tilde{\chi}\Big[-m_{21}(\theta-\theta^W)+m_{22}\bar{\sigma}_{nn}
	+m_{23}\Kn\pdfFrac{\bar{s}_j}{x_j}+m_{24}\Kn\pdfFrac{\bar{s}_{\langle n}}{x_{n\rangle}}+m_{25}\Kn\pdfFrac{u_{\langle n}}{x_{n\rangle}}\Big]-m_{26}\bar{s}_n.
    \end{aligned}
\end{equation*}
Substituting the proportionality relation \eqref{eq:C4_def} into the above expression yields:
\begin{equation*}
    \begin{aligned}
    m_{27}\Kn\pdfFrac{\theta}{x_n} - m_{28}\Kn\pdfFrac{\bar{\sigma}_{nj}}{x_j}
    &= \tilde{\chi}\Bigg[-m_{21}(\theta-\theta^W)+m_{22}\bar{\sigma}_{nn} \\
    &\quad + C_4 \Kn \left( m_{13}\pdfFrac{\bar{s}_j}{x_j} + m_{14}\pdfFrac{\bar{s}_{\langle n}}{x_{n\rangle}} + m_{15}\pdfFrac{u_{\langle n}}{x_{n\rangle}} \right)\Bigg] - m_{26}\bar{s}_n.
    \end{aligned}
\end{equation*}
We substitute the rearranged boundary condition \eqref{eq:m1eq_deri_left} into the equation above, leading to:
\begin{equation}\label{eq:m27_rearranged}
    \begin{aligned}
        &m_{27}\Kn\pdfFrac{\theta}{x_n} - m_{28}\Kn\pdfFrac{\bar{\sigma}_{nj}}{x_j}\\
        =& \tilde{\chi}\Big[ (-m_{21}+C_4 m_{11})(\theta-\theta^W) + (m_{22}+C_4 m_{12})\bar{\sigma}_{nn} \Big] - (C_4 + m_{26})\bar{s}_n.
    \end{aligned}
\end{equation}

Substituting this result back into the boundary terms of \eqref{eq:energy_weak_scaled} and collecting coefficients gives:
\begin{equation*}
    R_1 = \frac{3}{2}C_5(C_4 + m_{26}), \qquad
    T_2 = -\frac{3}{2}C_5\tilde{\chi}(m_{22}+C_4 m_{12}), \qquad
    S_4 = \frac{3}{2}C_5\tilde{\chi}(m_{21} - C_4 m_{11}).
\end{equation*}
The wall temperature term contributes $S_4 \smallip{\theta^W}{\gamma}_b$ to the linear functional $L_2(\gamma)$.

\subsection{Stress balance}
Recall that the strong form of the stress balance equation reads:
\begin{equation*}
    \begin{aligned}
    \partial_t \bar{\boldsymbol{\sigma}} &= -3k_2 \mathrm{Kn} (\nabla^2 \theta)_{\mathrm{stf}} - 2k_5 (\nabla \boldsymbol{u})_{\mathrm{stf}} - \frac{4}{5} k_8 (\nabla \bar{\boldsymbol{s}})_{\mathrm{stf}} \\
    &+ 2k_9 \mathrm{Kn} \nabla \cdot (\nabla \bar{\boldsymbol{\sigma}})_{\mathrm{stf}} + k_{10} \mathrm{Kn} \big(\nabla (\nabla \cdot \bar{\boldsymbol{\sigma}})\big)_{\mathrm{stf}} - \frac{l_2}{\mathrm{Kn}}\bar{\boldsymbol{\sigma}}.
    \end{aligned}
\end{equation*}
Multiplying both sides by the symmetric trace-free test tensor $\frac{1}{2}\bar{\bm{\tau}}$, integrating over the domain $\Omega$, and applying integration by parts yields:
\begin{equation}
    \begin{array}{ll}
        &\quad\frac{1}{2}\smallip{\partial_t \bar{\boldsymbol{\sigma}}}{\bar{\bm{\tau}}}
        = \frac{3}{2}k_2 \mathrm{Kn} \smallip{\nabla\theta}{\nabla\cdot\bar{\bm{\tau}}} - \frac{3}{2}k_2 \mathrm{Kn} \smallip{\pdfFrac{\theta}{x_n}}{\bar{\tau}_{nn}}_b - \frac{3}{2}k_2 \mathrm{Kn} \sum_{i=1}^2 \smallip{\pdfFrac{\theta}{x_{t_i}}}{\bar{\tau}_{nt_i}}_b \\
        &+ k_5 \smallip{\boldsymbol{u}}{\nabla\cdot\bar{\bm{\tau}}} - k_5 \sum_{i=1}^2 \smallip{u_{t_i}}{\bar{\tau}_{nt_i}}_b - k_5 \smallip{u_n}{\bar{\tau}_{nn}}_b \\
        &+ \frac{2}{5}k_8 \smallip{\bar{\boldsymbol{s}}}{\nabla\cdot\bar{\bm{\tau}}} - \frac{2}{5}k_8 \sum_{i=1}^2 \smallip{\bar{s}_{t_i}}{\bar{\tau}_{nt_i}}_b - \frac{2}{5}k_8 \smallip{\bar{s}_n}{\bar{\tau}_{nn}}_b \\
        &- k_9 \mathrm{Kn} \smallip{(\nabla \bar{\boldsymbol{\sigma}})_{\mathrm{stf}}}{(\nabla \bar{\bm{\tau}})_{\mathrm{stf}}} + \frac{3}{2}k_9 \mathrm{Kn} \smallip{\pdfFrac{\bar{\sigma}_{\langle nn}}{x_{n \rangle}}}{\bar{\tau}_{nn}}_b+2k_9\Kn \sum_{i=1}^2 \smallip{\pdfFrac{\bar{\sigma}_{\langle nt_i}}{x_{n\rangle}}}{\bar{\tau}_{nt_i}}_b \\
        &+ k_9 \sum_{i=1}^2 \mathrm{Kn} \smallip{\pdfFrac{\bar{\sigma}_{\langle t_i t_i}}{x_{n\rangle}} + \frac{1}{2}\pdfFrac{\bar{\sigma}_{\langle nn}}{x_{n\rangle}}}{\bar{\tau}_{t_i t_i} + \frac{1}{2}\bar{\tau}_{nn}}_b + 2k_9 \mathrm{Kn} \smallip{\pdfFrac{\bar{\sigma}_{\langle t_1 t_2}}{x_{n\rangle}}}{\bar{\tau}_{t_1 t_2}}_b \\
        &- \frac{1}{2}k_{10} \mathrm{Kn} \smallip{\nabla\cdot\bar{\boldsymbol{\sigma}}}{\nabla\cdot\bar{\bm{\tau}}} + \frac{1}{2}k_{10} \mathrm{Kn} \smallip{\pdfFrac{\bar{\sigma}_{nj}}{x_j}}{\bar{\tau}_{nn}}_b + \frac{1}{2}k_{10} \mathrm{Kn} \sum_{i=1}^2 \smallip{\pdfFrac{\bar{\sigma}_{t_i j}}{x_j}}{\bar{\tau}_{nt_i}}_b \\
        &- \frac{l_2}{2\mathrm{Kn}} \smallip{\bar{\boldsymbol{\sigma}}}{\bar{\bm{\tau}}}.
    \end{array}
    \label{eq:stress_weak_scaled}
\end{equation}
The boundary condition $u_n=0$ in \eqref{eq:BC_1} eliminates the corresponding boundary term. We substitute the boundary conditions \eqref{eq:BC_8}--\eqref{eq:BC_9} for the tangential-tangential components:
\begin{align*}
    &k_9 \mathrm{Kn} \smallip{\pdfFrac{\bar{\sigma}_{\langle t_i t_i}}{x_{n\rangle}} + \frac{1}{2}\pdfFrac{\bar{\sigma}_{\langle nn}}{x_{n\rangle}}}{\bar{\tau}_{t_i t_i} + \frac{1}{2}\bar{\tau}_{nn}}_b = -k_9 \tilde{\chi} m_{71} \smallip{\bar{\sigma}_{t_i t_i} + \frac{1}{2}\bar{\sigma}_{nn}}{\bar{\tau}_{t_i t_i} + \frac{1}{2}\bar{\tau}_{nn}}_b, \\
    &2k_9 \mathrm{Kn} \smallip{\pdfFrac{\bar{\sigma}_{\langle t_1 t_2}}{x_{n\rangle}}}{\bar{\tau}_{t_1 t_2}}_b = -2k_9 \tilde{\chi} m_{81} \smallip{\bar{\sigma}_{t_1 t_2}}{\bar{\tau}_{t_1 t_2}}_b.
\end{align*}
This defines the boundary coefficients:
\begin{displaymath}
    S_6 = k_9 \tilde{\chi} m_{71}, \quad S_8 = 2k_9 \tilde{\chi} m_{81}.
\end{displaymath}

\subsubsection{Normal-tangential boundary terms}
We collect all boundary terms involving $\bar{\tau}_{nt_i}$ from \eqref{eq:stress_weak_scaled}:
\begin{equation*}
    \mathcal{B}_{\text{stress}}^{nt} = \sum_{i=1}^2 \smallip{ - \frac{3}{2}k_2 \mathrm{Kn} \pdfFrac{\theta}{x_{t_i}} + 2k_9 \mathrm{Kn} \pdfFrac{\bar{\sigma}_{\langle t_i n}}{x_{n\rangle}} + \frac{1}{2}k_{10} \mathrm{Kn} \pdfFrac{\bar{\sigma}_{t_i j}}{x_j} - k_5 u_{t_i} -\frac{2}{5}k_8 \bar{s}_{t_i} }{ \bar{\tau}_{nt_i} }_b.
\end{equation*}
Using the proportionality coefficient \eqref{eq:C6_def} and substituting the boundary condition \eqref{eq:BC_4}, we obtain:
\begin{equation*}
    - \frac{3}{2}k_2 \mathrm{Kn} \pdfFrac{\theta}{x_{t_i}} + 2k_9 \mathrm{Kn} \pdfFrac{\bar{\sigma}_{\langle t_i n}}{x_{n \rangle}} + \frac{1}{2}k_{10} \mathrm{Kn} \pdfFrac{\bar{\sigma}_{t_i j}}{x_j} = \frac{1}{2}C_6 \left( m_{31}(u_{t_i}-u_{t_i}^W) + m_{32}\bar{s}_{t_i} - \frac{1}{\tilde{\chi}}\bar{\sigma}_{nt_i} \right).
\end{equation*}
Substituting back yields the coefficients: 
\begin{align*}
    S_7 = \frac{C_6}{2\tilde{\chi}}, 
    \quad R_4 = \frac{C_6 m_{31}}{2} - k_5, \quad R_2 = -\frac{C_6 m_{32}}{2} + \frac{2}{5}k_8.
\end{align*}

\subsubsection{Normal-normal boundary terms}
We collect boundary terms involving $\bar{\tau}_{nn}$:
\begin{equation*}
    \mathcal{B}_{\text{stress}}^{nn} = \smallip{- \frac{3}{2}k_2 \mathrm{Kn} \pdfFrac{\theta}{x_n} + \frac{3}{2}k_9 \mathrm{Kn} \pdfFrac{\bar{\sigma}_{\langle nn}}{x_{n\rangle}} + \frac{1}{2}k_{10} \mathrm{Kn} \pdfFrac{\bar{\sigma}_{nj}}{x_j} - \frac{2}{5}k_8 \bar{s}_n}{\bar{\tau}_{nn}}_b.
\end{equation*}
First, we rearrange \eqref{eq:BC_7} to get:
\begin{equation*}
\begin{split}
	& -\Kn \left( m_{67}\pdfFrac{\bar{\sigma}_{\langle nn}}{x_{n\rangle}} + m_{68}\pdfFrac{\bar{\sigma}_{nj}}{x_j} + m_{69}\pdfFrac{\theta}{x_n} \right)
	= m_{66} \bar{s}_n + \tilde{\chi} \big[-m_{61} (\theta -\theta^{W}) \\
    & \qquad +m_{62} \bar{\sigma}_{nn} + m_{63} \Kn \pdfFrac{\bar{s}_j}{x_j} +  m_{64} \Kn \pdfFrac{\bar{s}_{\langle n}}{x_{n \rangle}} + m_{65} \Kn \pdfFrac{u_{\langle n}}{x_{n \rangle}}\big]
\end{split}
\end{equation*}
Substituting the proportionality relation \eqref{eq:c2_def} and the rearranged boundary condition \eqref{eq:m1eq_deri_left} into the above expression yields:
\begin{equation}\label{eq:m67_rearranged}
    \begin{aligned}
    & -\Kn \left( m_{67}\pdfFrac{\bar{\sigma}_{\langle nn}}{x_{n\rangle}} + m_{68}\pdfFrac{\bar{\sigma}_{nj}}{x_j} + m_{69}\pdfFrac{\theta}{x_n} \right)
	= m_{66} \bar{s}_n + \tilde{\chi} \big[-m_{61} (\theta -\theta^{W}) \\
    & \qquad +m_{62} \bar{\sigma}_{nn} + C_7 \left( m_{11}(\theta-\theta^W) + m_{12}\bar{\sigma}_{nn} - \frac{1}{\tilde{\chi}} \bar{s}_n \right)\Bigg].
    \end{aligned}
\end{equation}
Next, we introduce the linear combination coefficients $\alpha_2, \beta_2$ satisfying the vector matching condition:
\begin{equation}
	\frac{1}{2}\begin{pmatrix}
		3k_2 \\ -3k_9 \\ -k_{10}
	\end{pmatrix}
	=
	\alpha_2
	\begin{pmatrix}
		m_{27} \\ 0 \\ -m_{28}
	\end{pmatrix}
	+
	\beta_2
	\begin{pmatrix}
		-m_{69} \\ -m_{67} \\ -m_{68}
	\end{pmatrix},
	\label{eq:nn_coeff_match}
\end{equation}
which allows us to rewrite the derivative combination in $\mathcal{B}_{\text{stress}}^{nn}$ as:
\begin{equation*}
	\begin{aligned}
		&- \frac{3}{2}k_2 \mathrm{Kn} \pdfFrac{\theta}{x_n} + \frac{3}{2}k_9 \mathrm{Kn} \pdfFrac{\bar{\sigma}_{\langle nn}}{x_{n\rangle}} + \frac{1}{2}k_{10} \mathrm{Kn} \pdfFrac{\bar{\sigma}_{nj}}{x_j} \\
		=\ &-\alpha_2 \left( m_{27}\mathrm{Kn}\pdfFrac{\theta}{x_n} - m_{28}\mathrm{Kn}\pdfFrac{\bar{\sigma}_{nj}}{x_j} \right)
		+\beta_2 \mathrm{Kn}\left( m_{67}\pdfFrac{\bar{\sigma}_{\langle nn}}{x_{n\rangle}} + m_{68}\pdfFrac{\bar{\sigma}_{nj}}{x_j} + m_{69}\pdfFrac{\theta}{x_n} \right).
	\end{aligned}
\end{equation*}
Substituting \eqref{eq:m27_rearranged} and \eqref{eq:m67_rearranged} into the above expression yields:
\begin{displaymath}
    - \frac{3}{2}k_2 \mathrm{Kn} \pdfFrac{\theta}{x_n} + \frac{3}{2}k_9 \mathrm{Kn} \pdfFrac{\bar{\sigma}_{\langle nn}}{x_{n\rangle}} + \frac{1}{2}k_{10} \mathrm{Kn} \pdfFrac{\bar{\sigma}_{nj}}{x_j} = - \left( Y_1(\theta-\theta^W) + Y_2 \bar{\sigma}_{nn} + Y_3 \bar{s}_n \right),
\end{displaymath}
where coefficients
\begin{gather*}
    Y_1 =\alpha_2\tilde{\chi}\left( C_4 m_{11} - m_{21} \right) + \beta_2\tilde{\chi}\left( C_7 m_{11}-m_{61} \right), \\
    Y_2 = \alpha_2\tilde{\chi}\left( C_4 m_{12} + m_{22} \right) + \beta_2\tilde{\chi}\left( C_7 m_{12}+m_{62} \right), \\ 
    Y_3 = -\alpha_2\left( m_{26}+C_4 \right) + \beta_2\left( m_{66}-C_7 \right).
\end{gather*}
Substituting back into $\mathcal{B}_{\text{stress}}^{nn}$ and combining the $\bar{s}_n$ term gives normal-normal boundary coefficients
\begin{displaymath}
    R_3 = Y_3 + \frac{2}{5}k_8, \quad \tilde{T}_2 = Y_1, \quad S_3 = Y_2.
\end{displaymath}
The wall temperature term contributes $\tilde{T}_2 \smallip{\theta^W}{\bar{\tau}_{nn}}_b$ to $L_3(\bar{\bm{\tau}})$.
\subsection{Momentum balance}
Recall that the strong form of the momentum balance equation reads:
\begin{equation*}
    \partial_t \boldsymbol{u} = -\nabla \rho - \nabla \theta + k_3 \mathrm{Kn} \nabla \cdot (\nabla \boldsymbol{u})_{\mathrm{stf}} + k_4 \mathrm{Kn} \nabla \cdot (\nabla \bar{\boldsymbol{s}})_{\mathrm{stf}} - k_5 \nabla \cdot \bar{\boldsymbol{\sigma}}.
\end{equation*}
Multiplying by the vector test function $\boldsymbol{v}$, integrating over $\Omega$, and applying integration by parts gives:
\begin{equation}
    \begin{aligned}
        \smallip{\partial_t \boldsymbol{u}}{\boldsymbol{v}}
        &= \smallip{\rho}{\nabla\cdot\boldsymbol{v}} + \smallip{\theta}{\nabla\cdot\boldsymbol{v}}
        - k_3 \mathrm{Kn} \smallip{(\nabla \boldsymbol{u})_{\mathrm{stf}}}{(\nabla \boldsymbol{v})_{\mathrm{stf}}} + k_3 \mathrm{Kn} \smallip{\pdfFrac{u_{\langle t_i}}{x_{n \rangle}}}{v_{t_i}}_b \\
        &- k_4 \mathrm{Kn} \smallip{(\nabla \bar{\boldsymbol{s}})_{\mathrm{stf}}}{(\nabla \boldsymbol{v})_{\mathrm{stf}}} + k_4 \mathrm{Kn} \smallip{\pdfFrac{\bar{s}_{\langle t_i}}{x_{n\rangle} }}{v_{t_i}}_b
        + k_5 \smallip{\nabla\cdot\bar{\boldsymbol{\sigma}}}{\boldsymbol{v}}.
    \end{aligned}
    \label{eq:momentum_weak_scaled}
\end{equation}
The essential boundary condition $u_n=0$ implies $v_n=0$ on $\partial\Omega$, so only tangential boundary terms remain.

We collect the tangential boundary terms:
\begin{equation}
    \mathcal{B}_{\text{momentum}}^t = \smallip{ k_3 \mathrm{Kn} \pdfFrac{u_{\langle t_i}}{x_{n \rangle}} + k_4 \mathrm{Kn} \pdfFrac{\bar{s}_{\langle t_i}}{x_{n \rangle}} - k_5 \bar{\sigma}_{nt_i} }{ v_{t_i}}_b.
\end{equation}
Using the linear combinations of boundary conditions \eqref{eq:BC_T4}--\eqref{eq:BC_T5} and the matching coefficients
\begin{equation*}
    \begin{pmatrix} \alpha_3 \\ \beta_3 \end{pmatrix} = \begin{pmatrix} m_{47} & m_{57} \\ m_{48} & m_{58} \end{pmatrix}^{-1} \begin{pmatrix} k_3 \\ k_4 \end{pmatrix},
\end{equation*}
we substitute the derivative terms and obtain:
\begin{equation*}
    k_3 \mathrm{Kn} \pdfFrac{u_{\langle t_i}}{x_{n \rangle}} + k_4 \mathrm{Kn} \pdfFrac{\bar{s}_{\langle t_i}}{x_{n \rangle}} = (\alpha_3 A_1 + \beta_3 B_1) (u_{t_i} - u^W_{t_i}) + (\alpha_3 A_2 + \beta_3 B_2) \bar{s}_{t_i} + (\alpha_3 A_3 + \beta_3 B_3) \bar{\sigma}_{nt_i},
\end{equation*}
where coefficients $A_i$ and $B_i$ come from \eqref{eq:A13B13}. Therefore, it holds that
\begin{align*}
    S_2 = -(\alpha_3 A_1 + \beta_3 B_1), \quad \tilde{T}_1 = -(\alpha_3 A_2 + \beta_3 B_2), \quad R_4 = -(\alpha_3 A_3 + \beta_3 B_3).
\end{align*}
The wall velocity term contributes $S_2 \smallip{u_{t_i}^W}{v_{t_i}}_b$ to the linear functional $L_4(\boldsymbol{v})$.

\subsection{Mass balance}
The strong form of the mass conservation equation reads $\partial_t \rho = -\nabla \cdot \boldsymbol{u}$. 
Multiplying by the scalar test function $q$, integrating over $\Omega$ yields:
\begin{equation*}
    \smallip{\partial_t \rho}{q} = -\smallip{\nabla \cdot \boldsymbol{u}}{q}.
\end{equation*}

\subsection{Result}
\begin{remark}[Consistency of the boundary coupling coefficients]
\label{rem:consist_of_paras}
Let
$Q_{\mathrm{bdry}}(\bcal{S}):=\mathscr{A}_{\mathrm{bdry}}(\bcal{S},\bcal{S})$
denote the boundary contribution in the diagonal evaluation of the weak formulation.
By the kinetic representation,
\[
Q_{\mathrm{bdry}}(\bcal{S})
=
-2\bigl(\widetilde Q A_n f_{\mathrm{even}},\,A_n f_{\mathrm{even}}\bigr)_b .
\]
Hence \(Q_{\mathrm{bdry}}(\bcal{S})\) is generated solely by the even part of the boundary
trace of the distribution function.

With respect to the normal reflection \(\xi_n\mapsto -\xi_n\), the trace moments$
\{\theta,\sigma_{nn},s_{t_i},u_{t_i}\}$
are even, whereas $\{s_n,\sigma_{nt_i}\}$ are odd. On the other hand, the coefficients $\{R_i\}$ appear in the bilinear boundary
couplings
\[
(r_n,\sigma_{nn})_b,\qquad
(r_{t_i},\sigma_{nt_i})_b,\qquad
(\theta,r_n)_b,\qquad
(\sigma_{nt_i},v_{t_i})_b .
\]
After diagonal evaluation \(\bcal{R}=\bcal{S}\), these become
\[
(s_n,\sigma_{nn})_b,\qquad
(s_{t_i},\sigma_{nt_i})_b,\qquad
(\theta,s_n)_b,\qquad
(\sigma_{nt_i},u_{t_i})_b .
\]
Each of these terms couples one odd trace moment with one even trace moment.

Therefore such odd-even boundary cross terms cannot occur in
\(Q_{\mathrm{bdry}}(S)\). Consequently, the formulas for the coefficients \(R_i\)
obtained from different component equations cannot be independent: they must
coincide exactly. Otherwise, after setting the test and trial variables equal, a
nonzero residual odd-even boundary cross term would remain in
\(Q_{\mathrm{bdry}}(\bcal{S})\), contradicting the kinetic representation above.

From the detailed expression for boundary conditions in \cite{lin2025time}, we can varify $T_1=\tilde{T}_1$ and $T_2=\tilde{T}_2$. So the boundary coupling coefficients are consistent.


\end{remark}
Mark the coefficients:
\begin{equation}
\begin{aligned}
    &\begin{bmatrix}
        k_4\\
        \frac{24}{25}k_7
    \end{bmatrix}
    =
    \begin{bmatrix}
        m_{47}&m_{57}\\
        m_{48}&m_{58}\\
    \end{bmatrix}
    \begin{bmatrix}
        \alpha_1\\
        \beta_1\\
    \end{bmatrix},\quad \frac{1}{2}\begin{bmatrix}
        3k_2\\
        -3k_9\\
        -k_{10}
    \end{bmatrix}=\begin{bmatrix}
        m_{27}& -m_{69}\\
        0& -m_{67}\\
        -m_{28}& -m_{68}\\
    \end{bmatrix}
    \begin{bmatrix}
      \alpha_2\\
      \beta_2\\
    \end{bmatrix},\\
    &\begin{bmatrix}
        k_3\\
        k_4\\
    \end{bmatrix}
    =\begin{bmatrix}
        m_{47}&m_{57}\\
        m_{48}&m_{58}\\
    \end{bmatrix}
    \begin{bmatrix}
        \alpha_3\\
        \beta_3\\
    \end{bmatrix},
\end{aligned}
\end{equation}
we have:
\begin{equation}
\begin{cases}
    \label{eq:parameters_Si}
    R_1 \coloneqq \frac{2}{5}C_1 m_{11} - k_0, \qquad 
    R_2 \coloneqq \alpha_1(C_2 - m_{46}) + \beta_1(C_3 - m_{56}), \\
    R_3 \coloneqq \frac{2}{5}C_1 m_{12}, \qquad 
    R_4 \coloneqq \frac{C_6 m_{31}}{2} - k_5, \\
    S_1 \coloneqq \tilde{\chi} \left( \alpha_1(m_{42} + C_2 m_{32}) + \beta_1(m_{52} + C_3 m_{32}) \right), \\
    S_2 \coloneqq \tilde{\chi}\left(\alpha_3(C_2 m_{31} - m_{41}) + \beta_3(C_3 m_{31} - m_{51})\right), \\
    S_3 \coloneqq \tilde{\chi} \left( \alpha_2(m_{22} + C_4 m_{12}) + \beta_2(m_{62} + C_7 m_{12}) \right), \\
    S_4 \coloneqq \frac{3}{2}C_5\tilde{\chi}(m_{21} - C_4 m_{11}), \qquad 
    S_5 \coloneqq \frac{2 C_1}{5 \tilde{\chi}}, \\
    S_6 \coloneqq k_9\tilde{\chi}m_{71}, \qquad 
    S_7 \coloneqq \frac{C_6}{2\tilde{\chi}}, \qquad 
    S_8 \coloneqq 2k_9\tilde{\chi}m_{81}, \\
    T_1 \coloneqq \tilde{\chi}\left(\alpha_1(C_2 m_{31} - m_{41}) + \beta_1(C_3 m_{31} - m_{51})\right),\quad T_2 \coloneqq -\frac{3}{2}C_5\tilde{\chi}(m_{22}+C_4 m_{12}).\\
\end{cases}
\end{equation}


\bibliographystyle{siamplain}
\bibliography{references}